\documentclass[a4paper,11pt]{amsart} % poner 'twoside' dentro de los [] si imprimo a doble cara.
\usepackage[letterpaper,left=2.5cm,right=2.5cm,top=4cm,bottom=2.5cm]{geometry}
\usepackage[english,activeacute]{babel} %cambiar a 'english' para que todos los titulos sean en inglés. NOTA: el cambio puede tirar warnings. Basta compilar de nuevo y listo.
\usepackage{color}
\usepackage{graphics}
\usepackage{graphicx}
\usepackage{hyperref}
\usepackage{longtable}
\usepackage{amsmath}
\usepackage{epigraph}
\usepackage{indentfirst}
\usepackage{verbatim}
\usepackage{array}
\usepackage{amssymb}
\usepackage{subfigure}
\usepackage{fancyhdr}
\usepackage{enumitem}
\DeclareMathAlphabet{\mathpzc}{OT1}{pzc}{m}{it} %\mathpzc{N}
\usepackage{mathrsfs}
\usepackage[small,bf]{caption}
\usepackage{cancel}
\usepackage{braket} %For quantum mechanics
\usepackage{mathrsfs}  %for mathscr
\usepackage{epstopdf}
\usepackage{amsthm}
\usepackage{dsfont}
\usepackage{mathtools} % para el declarepaireddelimiter
\DeclarePairedDelimiter\abs{\lvert}{\rvert}%
\newcommand{\R}{\mathbb{R}}

\newcommand{\Z}{\mathbb{Z}}
\newcommand{\E}{\mathbb{E}}
\newcommand{\T}{\mathbb{T}}
\newcommand{\myrightarrow}[1]{\xrightarrow{\makebox[2em][c]{$\scriptstyle#1$}}}
\DeclarePairedDelimiter\norm{\lVert}{\rVert}%
\DeclareMathOperator\supp{supp}

\newtheorem{theorem}{Theorem}[section]

\newtheorem{lemma}[theorem]{Lemma}
\newtheorem{proposition}[theorem]{Proposition}
\theoremstyle{definition}
\newtheorem{definition}[theorem]{Definition}
\newtheorem{remark}{Remark}

\title{Scaling limit of stationary coupled Sasamoto-Spohn models}
\author[Ian Butelmann and Gregorio R. Moreno Flores]{Ian Butelmann$^1$ and Gregorio R. Moreno Flores$^{2}$}

\address{Facultad de Matem\'aticas\\
Pontificia Universidad Cat\'olica de Chile\\
Vicu\~na Mackenna 4860, Macul\\
Santiago, Chile}

\email{ibutelmann@uc.cl, grmoreno@mat.uc.cl}

\thanks{$^{1,2}$ Facultad de Matem\'aticas, Pontificia Universidad Cat\'olica de Chile.}

\thanks{$^{1}$  Partially supported by Fondecyt grant 1171257}

\thanks{$^2$  Partially supported by Fondecyt grant 1211189, N\'ucleo Milenio `Modelos Estoc\'asticos de Sistemas Complejos y Desordenados' and MATH Amsud `Random Structures and Processes in Statistical Mechanics'}

\begin{document}

\begin{abstract}
	We introduce a family of stationary coupled Sasamoto-Spohn models and show that, in the weakly asymmetric regime, they converge to the energy solution of coupled Burgers equations. Moreover, we show that any system of coupled Burgers equations satisfying the so-called trilinear condition ensuring stationarity can be obtained as the scaling limit of a suitable system of coupled  Sasamoto-Spohn models.
	
	The core of our proof, which avoids the use of spectral gap estimates, consists in a second order Boltzmann-Gibbs principle for the discrete model. 
\end{abstract}

\maketitle

\tableofcontents

%%%%%%%%%%%%%%%%%%%%%%%%%%%%%%%%%%%%%%%%%%%%%%%%%%%%%%%%%%%%%%%
%%%%%%%%%%%%%%%%%%%%%%%%%%%%%%%%%%%%%%%%%%%%%%%%%%%%%%%%%%%%%%%
%%%%%%%%%%%%%%%%%%%%%%%%%%%%%%%%%%%%%%%%%%%%%%%%%%%%%%%%%%%%%%%
%%%%%%%%%%%%%%%%%%%%%%%%%%%%%%%%%%%%%%%%%%%%%%%%%%%%%%%%%%%%%%%
%%%%%%%%%%%%%%%%%%%%%%%%%%%%%%%%%%%%%%%%%%%%%%%%%%%%%%%%%%%%%%%
%%%%%%%%%%%%%%%%%%%%%%%%%%%%%%%%%%%%%%%%%%%%%%%%%%%%%%%%%%%%%%%
%%%%%%%%%%%%%%%%%%%%%%%%%%%%%%%%%%%%%%%%%%%%%%%%%%%%%%%%%%%%%%%
%%%%%%%%%%%%%%%%%%%%%%%%%%%%%%%%%%%%%%%%%%%%%%%%%%%%%%%%%%%%%%%

\section{Introduction, model and results}

%%%%%%%%%%%%%%%%%%%%%%%%%%%%%%%%%%%%%%%%%%%%%%%%%%%%%%%%%%%%%%%
%%%%%%%%%%%%%%%%%%%%%%%%%%%%%%%%%%%%%%%%%%%%%%%%%%%%%%%%%%%%%%%
%%%%%%%%%%%%%%%%%%%%%%%%%%%%%%%%%%%%%%%%%%%%%%%%%%%%%%%%%%%%%%%
%%%%%%%%%%%%%%%%%%%%%%%%%%%%%%%%%%%%%%%%%%%%%%%%%%%%%%%%%%%%%%%

The purpose of this work is twofold. First, we introduce a spatial discretization of coupled Burgers equations with an explicit invariant distribution. Second, we show that, at stationarity and in the weakly asymmetric regime, this model converges to the energy solution of coupled Burgers equations.

Respect to the first point, the most elementary spatial discretization of the single component Burgers equations is known as the Sasamoto-Spohn model \cite{KrS, LamAndShin, SasamotoSpohn}. Our model consists in a generalization to the multi-components setting. As we discuss below, the crucial point consists in providing a careful definition of the non-linear term at the discrete level to obtain a tractable invariant distribution (see Proposition \ref{thm:invariant-distribution} below). 

Respect to the second point, the convergence of the Sasamoto-Spohn model to the Burgers equation in the weakly asymmetric regime was shown in \cite{GP-KPZ} and \cite{jaraMoreno}. Our main result Theorem \ref{mainResult} is a generalization of the  second of these references and, as such, belongs to a long tradition of convergence results inside the KPZ universality class, dating back at least to the seminal work of Bertini and Giacomin \cite{BG}. We refer the reader to \cite{Corwin, Q} for extensive reviews of the literature on the KPZ universality class, including convergence theorems for everal models.

We start with a brief presentation of coupled Burgers equations in Section \ref{sec:coupledBurgers}. Then, we introduce our model and state our main result in Section \ref{ssmodel}.

%%%%%%%%%%%%%%%%%%%%%%%%%%%%%%%%%%%%%%%%%%%%%%%%%%%%%%%%%%%%%%%
%%%%%%%%%%%%%%%%%%%%%%%%%%%%%%%%%%%%%%%%%%%%%%%%%%%%%%%%%%%%%%%
%%%%%%%%%%%%%%%%%%%%%%%%%%%%%%%%%%%%%%%%%%%%%%%%%%%%%%%%%%%%%%%
%%%%%%%%%%%%%%%%%%%%%%%%%%%%%%%%%%%%%%%%%%%%%%%%%%%%%%%%%%%%%%%

\subsection{Coupled Burgers equations}
\label{sec:coupledBurgers}

%%%%%%%%%%%%%%%%%%%%%%%%%%%%%%%%%%%%%%%%%%%%%%%%%%%%%%%%%%%%%%%
%%%%%%%%%%%%%%%%%%%%%%%%%%%%%%%%%%%%%%%%%%%%%%%%%%%%%%%%%%%%%%%
%%%%%%%%%%%%%%%%%%%%%%%%%%%%%%%%%%%%%%%%%%%%%%%%%%%%%%%%%%%%%%%
%%%%%%%%%%%%%%%%%%%%%%%%%%%%%%%%%%%%%%%%%%%%%%%%%%%%%%%%%%%%%%%

We consider systems of coupled Burgers equations i.e. processes $u = (u_{1},\dots,u_{K})$ with $K \geq 1$ components satisfying systems of stochastic partial differential equations of the form
\begin{equation}
\label{csbeFunaki}
        \partial_{t} u_{k} = \frac{1}{2} \partial_{x}^{2} u_{k} + \sum_{i,j=1}^{K} \Gamma_{i,j}^{k} \partial_{x} (u_{i} u_{j}) + \partial_{x} \mathscr{W}_{k}, \quad 1 \leq k \leq K,
\end{equation}
where $(\mathscr{W}_{k})_{k}$ is a family of independent space-time white noises. Here, the space variable runs either on $\R$ or on the one-dimensional torus $\T$.

Coupled Burgers equations first appeared in the physics literature at the beginning of the nineties in the study of random interfaces \cite{ertasKardar}. A few years later, they reappeared in several contexts, including dynamics of crystals \cite{Lahiri}, magnetohydrodynamics \cite{Yanase, fleischerDIamond, basuetal} and sedimenting suspensions \cite{levineetal}. Later, they where used heuristically to compute the asymptotics of correlations in anharmonic chains \cite{mendlSpohn, spohnAnharmonic} and multi-species particle systems \cite{ferrariSasamotoSpohn}. We point the reader to \cite{spohnAnharmonic} for a deeper review of the physics literature. As in the one-component case, coupled Burgers equations can be seen as the evolution of the slope in systems of coupled KPZ equations.

At the mathematical level, they share the same ill-posedness problems as the usual one layer Burgers equation, namely, the solutions are distribution-valued, which makes the non-linear terms ill-defined in principle. They were formalized in the framework of paracontrolled distributions \cite{GIP} in \cite{funakiHoshino} on the torus, where it was also proved that, under the so-called trilinear condition, 
\begin{equation}\label{eq:trilinear-condition}
    \Gamma_{i,j}^{k} = \Gamma_{j,i}^{k} = \Gamma_{k,j}^{i},
\end{equation} 
they admit the product of independent space white noise as an invariant distribution. Under this condition, the theory of energy solution, which is the point of view we adopt in this work, was developed in \cite{GP-generator}. Existence and uniqueness is shown on the torus but uniqueness on the whole line is still an open question. Recently, this theory was applied to show the convergence of stationary multi-species zero-range processes \cite{bernardinFunakiSethuraman}. 
We point the interested reader to \cite{Patricia-review} for a deeper discussion.

We will state the precise definition of energy solutions of coupled Burgers equations in Section \ref{sec:ES}.

%%%%%%%%%%%%%%%%%%%%%%%%%%%%%%%%%%%%%%%%%%%%%%%%%%%%%%%%%%%%%%%
%%%%%%%%%%%%%%%%%%%%%%%%%%%%%%%%%%%%%%%%%%%%%%%%%%%%%%%%%%%%%%%
%%%%%%%%%%%%%%%%%%%%%%%%%%%%%%%%%%%%%%%%%%%%%%%%%%%%%%%%%%%%%%%
%%%%%%%%%%%%%%%%%%%%%%%%%%%%%%%%%%%%%%%%%%%%%%%%%%%%%%%%%%%%%%%

\subsection{Coupled Sasamoto-Spohn models and main result}
\label{ssmodel}

%%%%%%%%%%%%%%%%%%%%%%%%%%%%%%%%%%%%%%%%%%%%%%%%%%%%%%%%%%%%%%%
%%%%%%%%%%%%%%%%%%%%%%%%%%%%%%%%%%%%%%%%%%%%%%%%%%%%%%%%%%%%%%%
%%%%%%%%%%%%%%%%%%%%%%%%%%%%%%%%%%%%%%%%%%%%%%%%%%%%%%%%%%%%%%%
%%%%%%%%%%%%%%%%%%%%%%%%%%%%%%%%%%%%%%%%%%%%%%%%%%%%%%%%%%%%%%%

%%%%%%%%%%%%%%%%%%%%%%%%%%%%%%%%%%%%%%%%%%%%%%%%%%%%%%%%%%%%%%%
%%%%%%%%%%%%%%%%%%%%%%%%%%%%%%%%%%%%%%%%%%%%%%%%%%%%%%%%%%%%%%%
%Usual Sasamoto-Spohn
%%%%%%%%%%%%%%%%%%%%%%%%%%%%%%%%%%%%%%%%%%%%%%%%%%%%%%%%%%%%%%%
%%%%%%%%%%%%%%%%%%%%%%%%%%%%%%%%%%%%%%%%%%%%%%%%%%%%%%%%%%%%%%%
First, we recall the one component Sasamoto-Spohn model. This is a spatial discretization of the stochastic Burgers equation (i.e. \eqref{csbeFunaki} with $K=1$) consisting of a system of coupled diffusions $u=(u_{j})_{j}$ that satisfy
\begin{equation*}
    d u_{j} = \frac{1}{2} \Delta u_{j} + \epsilon B_{j}(u) + d\xi_{j} - d \xi_{j-1}
\end{equation*}
where  the index $j$ runs either on $\Z$ or $\Z_{M} = \Z/M\Z$, $(\xi_{j})_{j}$ is an i.i.d. family of standard one-dimensional Brownian motions,
\begin{equation*}
    \Delta u_{j} = u_{j+1} + u_{j-1} - 2 u_{j},
\end{equation*}
is the discrete Laplacian
and $B(u)$ is a discretization of the non-linear term in Burgers equation defined as
\begin{equation*}
    B_{j}(u) = w_{j} - w_{j-1}
    \quad
    \text{ with } 
    \quad
    w_{j} = \frac{1}{3}(u_{j}^{2} + u_{j} u_{j+1} + u_{j+1}^{2}).
\end{equation*}
This model was introduced in \cite{KrS} (see also \cite{LamAndShin}) and further studied in \cite{SasamotoSpohn}.
It is known that the product of independent Gaussians is stationary. 
%This yields the explicit invariant measure $\mu = \rho^{\otimes \Z}$ or $\mu = \rho^{\otimes \Z_{M}}$ where $d\rho(x) = \frac{1}{\sqrt{2 \pi}} e^{-\frac{x^{2}}{2}}dx$. 
At a first sight, they would be much simpler ways to discretize the non-linear term. However, this precise definition seems to be the simplest one such that the product of Gaussians is invariant.
Convergence to the stochastic Burgers equation was obtained in \cite{GP-KPZ} and \cite{jaraMoreno} in the weakly-asymmetric regime i.e. scaling time by $n^2$, space by $n$ and taking $\epsilon = \epsilon_{n} = n^{-\frac{1}{2}}$.
%%%%%%%%%%%%%%%%%%%%%%%%%%%%%%%%%%%%%%%%%%%%%%%%%%%%%%%%%%%%%%%
%%%%%%%%%%%%%%%%%%%%%%%%%%%%%%%%%%%%%%%%%%%%%%%%%%%%%%%%%%%%%%%

%%%%%%%%%%%%%%%%%%%%%%%%%%%%%%%%%%%%%%%%%%%%%%%%%%%%%%%%%%%%%%%
%%%%%%%%%%%%%%%%%%%%%%%%%%%%%%%%%%%%%%%%%%%%%%%%%%%%%%%%%%%%%%%
%Our model
%%%%%%%%%%%%%%%%%%%%%%%%%%%%%%%%%%%%%%%%%%%%%%%%%%%%%%%%%%%%%%%
%%%%%%%%%%%%%%%%%%%%%%%%%%%%%%%%%%%%%%%%%%%%%%%%%%%%%%%%%%%%%%%
We consider the following generalization: let $u = (u_{1},\dots, u_{K})$ be a vector of stochastic processes, where each $u_{k} = (u_{k,j})_{j}$ satisifies
%%%%%%%%%%%%%%%%%%%%%%%%%%%%%%%%%%%%%%%%%%%%%%%%%%%%%%%%%%%%%%%
%%%%%%%%%%%%%%%%%%%%%%%%%%%%%%%%%%%%%%%%%%%%%%%%%%%%%%%%%%%%%%%
\begin{equation}
    \label{eq:SS-propio}
    d u_{k,j} = \frac{1}{2} \Delta u_{k,j} + \epsilon B_{k,j}(u) + d\xi_{k,j} - d\xi_{k,j-1}, \quad k \in \Z_{K}, \ j \in \Z_{M}.
\end{equation}
%%%%%%%%%%%%%%%%%%%%%%%%%%%%%%%%%%%%%%%%%%%%%%%%%%%%%%%%%%%%%%%
%%%%%%%%%%%%%%%%%%%%%%%%%%%%%%%%%%%%%%%%%%%%%%%%%%%%%%%%%%%%%%%
As above, $\Delta u_{k,j} = u_{k,j+1} + u_{k,j-1} - 2 u_{k,j}$ and $(\xi_{k,j})_{k,j}$ are independent Brownian motions. Now, $B_{k,j}(u)$ is a quadratic polynomial in the variables $u_{k,j}$ defined as
%%%%%%%%%%%%%%%%%%%%%%%%%%%%%%%%%%%%%%%%%%%%%%%%%%%%%%%%%%%%%%%
%%%%%%%%%%%%%%%%%%%%%%%%%%%%%%%%%%%%%%%%%%%%%%%%%%%%%%%%%%%%%%%
\begin{eqnarray*}
    B_{k,j} &=& G_{k,j} - G_{k,j-1},
    \\
    G_{k,j}  &=& \alpha_{k} w_{k,j} + \sum_{l \in \Z_{K}^{*}} \beta_{k}^{l} b_{k,j}^{l} + \sum_{l \in \Z_{K}^{*}} \gamma_{k}^{l} r_{k,j}^{l} +\sum_{l \in \Z_{K}^{*}} \sum_{\substack{l' \in \Z_{K}^{*} \\ l' \neq l}} \lambda_{k}^{k-l,k-l'} p_{k,j}^{l,l'},
\end{eqnarray*}
where $\Z^*_K = \Z_K \backslash \{0\}$, operations are considered modulo $K$, and
\begin{eqnarray*}
	w_{k,j} &=& \frac{1}{3}(u_{k,j}^{2} + u_{k,j}u_{k,j+1} + u_{k,j+1}^{2}),
	\\
	b_{k,j}^{l} &=& \frac{1}{2}(u_{k,j}u_{k+l,j} + u_{k,j+1}u_{k+l,j+1}),
	\\
	 r_{k,j}^{l} &=& u_{k-l,j}u_{k-l,j+1},
	 \\
	 p_{k,j}^{l,l'} &=& \frac{1}{6}( 2u_{k-l,j}u_{k-l',j} + u_{k-l,j}u_{k-l',j+1} + u_{k-l,j+1}u_{k-l',j} + 2u_{k-l,j+1}u_{k-l',j+1}).
\end{eqnarray*}
%%%%%%%%%%%%%%%%%%%%%%%%%%%%%%%%%%%%%%%%%%%%%%%%%%%%%%%%%%%%%%%
%%%%%%%%%%%%%%%%%%%%%%%%%%%%%%%%%%%%%%%%%%%%%%%%%%%%%%%%%%%%%%%
\begin{remark}
At first glance it would be easier to write $\lambda_{k}^{l,l'}$ instead of $\lambda_{k}^{k-l,k-l'}$ above, but this would complicate the correspondence between our model and the coupled Burgers equations \eqref{csbeFunaki}. This is discussed in Lemma \ref{thm:equivalence-trilinear} below.
\end{remark}
%%%%%%%%%%%%%%%%%%%%%%%%%%%%%%%%%%%%%%%%%%%%%%%%%%%%%%%%%%%%%%%
%%%%%%%%%%%%%%%%%%%%%%%%%%%%%%%%%%%%%%%%%%%%%%%%%%%%%%%%%%%%%%%
Our model can look overly complicated but, as in the single component case, the discretization of the nonlinear term has to be carefully performed in order to obtain a tractable invariant measure.
We make the following assumptions on the coefficients, 
\begin{align}
\label{betaGamma}
    \beta_{k}^{a} & = 2 \gamma_{k+a}^{a} \\
    \label{lambda}
    \lambda_{k}^{k-a,k-a'} & = \lambda_{k}^{k-a',k-a} = \lambda_{k-a}^{k,k-a'}
\end{align}
for all $k,a,a' \in \Z_{K}$.
%%%%%%%%%%%%%%%%%%%%%%%%%%%%%%%%%%%%%%%%%%%%%%%%%%%%%%%%%%%%%%%
%%%%%%%%%%%%%%%%%%%%%%%%%%%%%%%%%%%%%%%%%%%%%%%%%%%%%%%%%%%%%%%

%%%%%%%%%%%%%%%%%%%%%%%%%%%%%%%%%%%%%%%%%%%%%%%%%%%%%%%%%%%%%%%
%%%%%%%%%%%%%%%%%%%%%%%%%%%%%%%%%%%%%%%%%%%%%%%%%%%%%%%%%%%%%%%
Let $\mu_{K,M} = (\rho^{\otimes \Z_{K}})^{\otimes \Z_{M}}$, where $d\rho(x) = \frac{1}{\sqrt{2 \pi}} e^{-\frac{x^{2}}{2}}dx$.
We denote the density of $\mu_{K,M}$ by $\rho_{K,M}$.
%%%%%%%%%%%%%%%%%%%%%%%%%%%%%%%%%%%%%%%%%%%%%%%%%%%%%%%%%%%%%%%
%%%%%%%%%%%%%%%%%%%%%%%%%%%%%%%%%%%%%%%%%%%%%%%%%%%%%%%%%%%%%%%
\begin{proposition}\label{thm:invariant-distribution}
	Assume \eqref{betaGamma} and \eqref{lambda}.
	Then, the law $\mu_{K,M}$ defined above is invariant for the dynamics given by \eqref{eq:SS-propio}.
\end{proposition}
%%%%%%%%%%%%%%%%%%%%%%%%%%%%%%%%%%%%%%%%%%%%%%%%%%%%%%%%%%%%%%%
%%%%%%%%%%%%%%%%%%%%%%%%%%%%%%%%%%%%%%%%%%%%%%%%%%%%%%%%%%%%%%%
This is proved in Section \ref{sec:invariant-measure}.

%%%%%%%%%%%%%%%%%%%%%%%%%%%%%%%%%%%%%%%%%%%%%%%%%%%%%%%%%%%%%%%
%%%%%%%%%%%%%%%%%%%%%%%%%%%%%%%%%%%%%%%%%%%%%%%%%%%%%%%%%%%%%%%
We now state our main result. We fix $K\geq 1$. For each, $k\in\Z_K$, we define the fluctuation field $\mathcal{X}_{k,\cdot}^{n}$ acting on test functions $\varphi\in\mathcal{S}(\T)$ as
\begin{equation*}
    \mathcal{X}_{k,t}^{n}(\varphi) = \frac{1}{n^{\frac{1}{2}}} \sum_{j \in \Z_{M}} u_{k,j}(tn^2) \varphi_{j}^{n},
    \quad
    \text{where}
    \quad
    \varphi_{j}^{n} = \varphi\left(\frac{j}{n}\right).
\end{equation*}
 Let $\mathbb{P}_{n}$ be the law of the process $u = (u_{1}, \dots, u_{K})$ with $\epsilon_{n} = n^{-\frac{1}{2}}$ and initial law $\mu_{K,n}$. The notion of energy solution alluded in the Theorem below will be properly defined in Section \ref{sec:ES}.
%%%%%%%%%%%%%%%%%%%%%%%%%%%%%%%%%%%%%%%%%%%%%%%%%%%%%%%%%%%%%%%
%%%%%%%%%%%%%%%%%%%%%%%%%%%%%%%%%%%%%%%%%%%%%%%%%%%%%%%%%%%%%%%
%%%%%%%%%%%%%%%%%%%%%%%%%%%%%%%%%%%%%%%%%%%%%%%%%%%%%%%%%%%%%%%
%%%%%%%%%%%%%%%%%%%%%%%%%%%%%%%%%%%%%%%%%%%%%%%%%%%%%%%%%%%%%%%
\begin{theorem}
\label{mainResult}
	Assume \eqref{betaGamma} and \eqref{lambda} and let $T>0$. Then, if $\epsilon_{n} = n^{-\frac{1}{2}}$, the sequence $(\mathcal{X}^{n})_{n}$ converges 	in law in $C([0,T], \mathcal{S}'(\mathbb{T})^K)$ to the unique energy solution of the coupled stochastic Burgers equations
	\begin{equation}\label{eq:limit-Burgers}
  	 	\partial_{t} u_{k} 
   		= 
  	 	\frac{1}{2}\partial_{x}^{2} u_{k} 
   		+ 
   		\partial_{x} B_k(u)
   		+ 
   		\partial_{x} \mathscr{W}_{k},
   		\quad
   		k\in\Z_k,
	\end{equation}
	where $\mathscr{W}_{1},\dots,\mathscr{W}_{K}$ are independent white noises and
	\begin{equation}
   		B_k(u)
   		= 
   		\alpha_{k} u_{k}^{2} 
   		+ 
   		\sum_{l \in \Z_{K}^{*}} \left\{
   			\beta_{k}^{l} u_{k} u_{k+l} 
   			+ 
   			\gamma_{k}^{l} u_{k-l}^{2} 
   		\right\}
   		+ 
   		\sum_{l \in \Z_{K}^{*}} \sum_{\substack{l' \in \Z_{K}^{*} \\ l' \neq l}} \lambda_{k}^{k-l,k-l'} u_{k-l}u_{k-l'}.
\end{equation}
\end{theorem}
%%%%%%%%%%%%%%%%%%%%%%%%%%%%%%%%%%%%%%%%%%%%%%%%%%%%%%%%%%%%%%%
%%%%%%%%%%%%%%%%%%%%%%%%%%%%%%%%%%%%%%%%%%%%%%%%%%%%%%%%%%%%%%%
%%%%%%%%%%%%%%%%%%%%%%%%%%%%%%%%%%%%%%%%%%%%%%%%%%%%%%%%%%%%%%%
%%%%%%%%%%%%%%%%%%%%%%%%%%%%%%%%%%%%%%%%%%%%%%%%%%%%%%%%%%%%%%%
As usual in this setting, the central ingredient of the proof is a certain second order Boltzmann-Gibbs principle (see Theorem \ref{thm:BoltzmannGibbs} and \ref{thm:BoltzmannGibbsCrossed}). This technique originated in \cite{goncalvesJaraEnergysol} where energy solutions for the one-dimensional Burgers equation were introduced. Our approach to the proof of this result is inspired by \cite{BGGolcalvesJaraSimon}, \cite{jaraMoreno} and \cite{JaM2}, and avoids the use of spectral gap estimates. However, we adopt a slightly different though equivalent path to construct the quadratic term in the continuum which allows us to give a much simpler proof. This is briefly discussed in Remark \ref{rk:quadratic}. Convergence to coupled Burgers equations for multi-type zero range processes was obtained in the recent work \cite{bernardinFunakiSethuraman}.

%%%%%%%%%%%%%%%%%%%%%%%%%%%%%%%%%%%%%%%%%%%%%%%%%%%%%%%%%%%%%%%
%%%%%%%%%%%%%%%%%%%%%%%%%%%%%%%%%%%%%%%%%%%%%%%%%%%%%%%%%%%%%%%
\begin{remark}
Our techniques can be easily extended to the whole line $\Z$ with mainly notational modifications, showing tightness of the fluctuation field and that any limit point is an energy solution of the coupled Burgers equations. However, the uniqueness of energy solutions is not yet proved in this setting.
\end{remark}
%%%%%%%%%%%%%%%%%%%%%%%%%%%%%%%%%%%%%%%%%%%%%%%%%%%%%%%%%%%%%%%
%%%%%%%%%%%%%%%%%%%%%%%%%%%%%%%%%%%%%%%%%%%%%%%%%%%%%%%%%%%%%%%
%%%%%%%%%%%%%%%%%%%%%%%%%%%%%%%%%%%%%%%%%%%%%%%%%%%%%%%%%%%%%%%
%%%%%%%%%%%%%%%%%%%%%%%%%%%%%%%%%%%%%%%%%%%%%%%%%%%%%%%%%%%%%%%
Note that our way of writing the coupled Burgers equations in \eqref{eq:limit-Burgers} is equivalent to \eqref{csbeFunaki}
once we make the identification
\begin{equation*}
     \alpha_{k} = \Gamma_{k,k}^{k}, \quad \frac{\beta_{k}^{l}}{2} = \Gamma_{k,k+l}^{k}, \quad \gamma_{k}^{l} = \Gamma_{k-l,k-l}^{k}, \quad  \lambda_{k}^{k-l,k-l'} = \Gamma_{k-l,k-l'}^{k}.
\end{equation*}
Our assumptions on the coefficients of the model are in fact equivalent to the trilinear condition.
%%%%%%%%%%%%%%%%%%%%%%%%%%%%%%%%%%%%%%%%%%%%%%%%%%%%%%%%%%%%%%%
%%%%%%%%%%%%%%%%%%%%%%%%%%%%%%%%%%%%%%%%%%%%%%%%%%%%%%%%%%%%%%%
\begin{lemma}\label{thm:equivalence-trilinear}
The trilinear condition \eqref{eq:trilinear-condition} and conditions \eqref{betaGamma}-\eqref{lambda} are equivalent.
\end{lemma}
%%%%%%%%%%%%%%%%%%%%%%%%%%%%%%%%%%%%%%%%%%%%%%%%%%%%%%%%%%%%%%%
%%%%%%%%%%%%%%%%%%%%%%%%%%%%%%%%%%%%%%%%%%%%%%%%%%%%%%%%%%%%%%%
The proof is deferred to Appendix \ref{sec:proof-trilinear}.
%%%%%%%%%%%%%%%%%%%%%%%%%%%%%%%%%%%%%%%%%%%%%%%%%%%%%%%%%%%%%%%
%%%%%%%%%%%%%%%%%%%%%%%%%%%%%%%%%%%%%%%%%%%%%%%%%%%%%%%%%%%%%%%
In particular, this show that any system of coupled Burgers equations satisfying the trilinear condition can be obtained as the weakly asymmetric limit of suitable coupled Sasamoto-Spohn models.

%%%%%%%%%%%%%%%%%%%%%%%%%%%%%%%%%%%%%%%%%%%%%%%%%%%%%%%%%%%%%%%
%%%%%%%%%%%%%%%%%%%%%%%%%%%%%%%%%%%%%%%%%%%%%%%%%%%%%%%%%%%%%%%
%%%%%%%%%%%%%%%%%%%%%%%%%%%%%%%%%%%%%%%%%%%%%%%%%%%%%%%%%%%%%%%
%%%%%%%%%%%%%%%%%%%%%%%%%%%%%%%%%%%%%%%%%%%%%%%%%%%%%%%%%%%%%%%

\subsection{Structure of the article}

%%%%%%%%%%%%%%%%%%%%%%%%%%%%%%%%%%%%%%%%%%%%%%%%%%%%%%%%%%%%%%%
%%%%%%%%%%%%%%%%%%%%%%%%%%%%%%%%%%%%%%%%%%%%%%%%%%%%%%%%%%%%%%%
%%%%%%%%%%%%%%%%%%%%%%%%%%%%%%%%%%%%%%%%%%%%%%%%%%%%%%%%%%%%%%%
%%%%%%%%%%%%%%%%%%%%%%%%%%%%%%%%%%%%%%%%%%%%%%%%%%%%%%%%%%%%%%%

We start by providing the precise definition of energy solutions in Section \ref{sec:ES}. Then, we introduce the coupled Sasamoto-Spohn models along with some of their basic properties in Section \ref{sec:coupled-SS}. Our core estimates are proved in Section \ref{sec:Estimates}. In particular, the second order Boltzmann-Gibbs principle is proved in Section \ref{subsec:boltzmann-gibbs}. As far as we know, we provide the simplest proof of such a result in the literature.

 In Section \ref{sec:tightness}, we prove the tightness of the fluctuating field by showing the tightness of each term in the martingale decomposition given in Section \ref{sec:MartingaleDescomposition}. In Section \ref{sec:identification}, we identify the unique limit point.
Finally, the two appendices contain the proofs of the equivalence between conditions \eqref{betaGamma}-\eqref{lambda} and the trilinear condition, and a key identity in the proof of stationarity, respectively.
%%%%%%%%%%%%%%%%%%%%%%%%%%%%%%%%%%%%%%%%%%%%%%%%%%%%%%%%%%%%%%%
%%%%%%%%%%%%%%%%%%%%%%%%%%%%%%%%%%%%%%%%%%%%%%%%%%%%%%%%%%%%%%%
%%%%%%%%%%%%%%%%%%%%%%%%%%%%%%%%%%%%%%%%%%%%%%%%%%%%%%%%%%%%%%%
%%%%%%%%%%%%%%%%%%%%%%%%%%%%%%%%%%%%%%%%%%%%%%%%%%%%%%%%%%%%%%%

\subsection{General Notations}

%%%%%%%%%%%%%%%%%%%%%%%%%%%%%%%%%%%%%%%%%%%%%%%%%%%%%%%%%%%%%%%
%%%%%%%%%%%%%%%%%%%%%%%%%%%%%%%%%%%%%%%%%%%%%%%%%%%%%%%%%%%%%%%
%%%%%%%%%%%%%%%%%%%%%%%%%%%%%%%%%%%%%%%%%%%%%%%%%%%%%%%%%%%%%%%
%%%%%%%%%%%%%%%%%%%%%%%%%%%%%%%%%%%%%%%%%%%%%%%%%%%%%%%%%%%%%%%

We work on the torus $\T = \R/\Z$.
At the discrete level, we write $\Z_m = \Z/m\Z$ and $\Z_m^*=\Z_m\backslash \{0\}$.
We denote by $\mathcal{S}(\T)$ the space of Schwartz test functions on $\T$ and by $\mathcal{S}'(\T)$ the space of tempered distributions.
 For $n \geq 1$ and a smooth function $\varphi$, we define 
 $\varphi_{j}^{n} = \varphi(\frac{j}{n})$, 
 $\nabla^{n} \varphi_{j}^{n} = n (\varphi_{j+1}^{n} - \varphi_{j}^{n})$ 
 and 
 $\Delta^{n} \varphi_{j}^{n} = n^2(\varphi_{j+1}^{n} + \varphi_{j-1}^{n} - 2\varphi_{j}^{n}) $. We also define
\begin{equation*}
    \mathcal{E}(\varphi) = \int_{\T} \varphi^{2}(x)dx, \quad \mathcal{E}_{n}(\psi) = \frac{1}{n} \sum_{j \in \Z_{M}} \psi_{j}^{2},
\end{equation*}
for $\varphi \in L^{2}(\T)$ and $\psi \in l^{2}(\Z_{M})$ respectively and, with a slight abuse of notation,
\begin{equation*}
    \mathcal{E}_{n}(\varphi) = \frac{1}{n} \sum_{j \in \Z_{M}} (\varphi_{j}^{n})^{2},
\end{equation*}
for $\varphi \in L^{2}(\T)$

We denote by $\mathscr{C}$ the space of twice continuously differentiable functions from $\R^{KM}$ to $\R$ with polynomial growth of their derivatives up to order two, where we identify $\R^{KM}$ with $\R^{\Z_{K} \times \Z_{M}}$.
For a function $g: \R^{KM} \to \R$ we denote by $\supp{(g)}$ the smallest set $S \subseteq \Z_{K} \times \Z_{M}$ such that
\begin{equation*}
    g(u) = g(\Tilde{u}),
\end{equation*}
if $u_{k,j} = \Tilde{u}_{k,j}$ for all $(k,j) \in S$.

We denote by $\mathbb{P}_{n}$ the law of the process $u = (u_{1}, \dots, u_{K})$ with $\epsilon_{n} = n^{-\frac{1}{2}}$ and we let $\E_{n}$ denote expectation with respect to $\mathbb{P}_{n}$.

As usual, $C$ denotes a constant which value can change from line to line.

%%%%%%%%%%%%%%%%%%%%%%%%%%%%%%%%%%%%%%%%%%%%%%%%%%%%%%%%%%%%%%%
%%%%%%%%%%%%%%%%%%%%%%%%%%%%%%%%%%%%%%%%%%%%%%%%%%%%%%%%%%%%%%%
%%%%%%%%%%%%%%%%%%%%%%%%%%%%%%%%%%%%%%%%%%%%%%%%%%%%%%%%%%%%%%%
%%%%%%%%%%%%%%%%%%%%%%%%%%%%%%%%%%%%%%%%%%%%%%%%%%%%%%%%%%%%%%%

%%%%%%%%%%%%%%%%%%%%%%%%%%%%%%%%%%%%%%%%%%%%%%%%%%%%%%%%%%%%%%%
%%%%%%%%%%%%%%%%%%%%%%%%%%%%%%%%%%%%%%%%%%%%%%%%%%%%%%%%%%%%%%%
%%%%%%%%%%%%%%%%%%%%%%%%%%%%%%%%%%%%%%%%%%%%%%%%%%%%%%%%%%%%%%%
%%%%%%%%%%%%%%%%%%%%%%%%%%%%%%%%%%%%%%%%%%%%%%%%%%%%%%%%%%%%%%%
%%%%%%%%%%%%%%%%%%%%%%%%%%%%%%%%%%%%%%%%%%%%%%%%%%%%%%%%%%%%%%%
%%%%%%%%%%%%%%%%%%%%%%%%%%%%%%%%%%%%%%%%%%%%%%%%%%%%%%%%%%%%%%%
%%%%%%%%%%%%%%%%%%%%%%%%%%%%%%%%%%%%%%%%%%%%%%%%%%%%%%%%%%%%%%%
%%%%%%%%%%%%%%%%%%%%%%%%%%%%%%%%%%%%%%%%%%%%%%%%%%%%%%%%%%%%%%%

\section{Energy Solutions of the coupled stochastic Burgers equation}
\label{sec:ES}

%%%%%%%%%%%%%%%%%%%%%%%%%%%%%%%%%%%%%%%%%%%%%%%%%%%%%%%%%%%%%%%
%%%%%%%%%%%%%%%%%%%%%%%%%%%%%%%%%%%%%%%%%%%%%%%%%%%%%%%%%%%%%%%
%%%%%%%%%%%%%%%%%%%%%%%%%%%%%%%%%%%%%%%%%%%%%%%%%%%%%%%%%%%%%%%
%%%%%%%%%%%%%%%%%%%%%%%%%%%%%%%%%%%%%%%%%%%%%%%%%%%%%%%%%%%%%%%
The theory of energy solutions for the (one-layer) Burgers equation originated in \cite{goncalvesJaraEnergysol} and was subsequently developed in \cite{GuJ}. Uniqueness was proved in \cite{GP-uniqueness} on the whole line. The corresponding result in the multi-component setting on the one-dimensional torus was provided in \cite{GP-generator}. Below, we follow the exposition given in \cite{bernardinFunakiSethuraman}, which is equivalent (see \cite{BGGolcalvesJaraSimon} for the one-component case). 

The theory of energy solutions has been extremely successful to show convergence of stationary discrete models to the Burgers equation. See for instance \cite{goncalvesJaraEnergysol, GJS, DGP, jaraMoreno, JaM2}.
	
%%%%%%%%%%%%%%%%%%%%%%%%%%%%%%%%%%%%%%%%%%%%%%%%%%%%%%%%%%%%%%%
%%%%%%%%%%%%%%%%%%%%%%%%%%%%%%%%%%%%%%%%%%%%%%%%%%%%%%%%%%%%%%%
%%%%%%%%%%%%%%%%%%%%%%%%%%%%%%%%%%%%%%%%%%%%%%%%%%%%%%%%%%%%%%%
%%%%%%%%%%%%%%%%%%%%%%%%%%%%%%%%%%%%%%%%%%%%%%%%%%%%%%%%%%%%%%%

We recall the system of coupled stochastic Burgers equation,
\begin{equation}
\label{coupledBurgers}
        \partial_{t} u_{k} = \frac{1}{2} \partial_{x}^{2} u_{k} + \sum_{i,j=1}^{K} \Gamma_{i,j}^{k} \partial_{x} (u_{i} u_{j}) + \partial_{x} \mathscr{W}_{k}, \quad 1 \leq k \leq K,
\end{equation}
where the initial condition is taken as the product of independent space white noises on the one-dimensional torus.

We start with two definitions:
%%%%%%%%%%%%%%%%%%%%%%%%%%%%%%%%%%%%%%%%%%%%%%%%%%%%%%%%%%%%%%%
%%%%%%%%%%%%%%%%%%%%%%%%%%%%%%%%%%%%%%%%%%%%%%%%%%%%%%%%%%%%%%%
\begin{definition}
We say that a process $\{u_{t}=(u_{1,t},\cdots,u_{K,t}): \ t \in [0,T] \}$ taking values in $C([0,T],\mathcal{S}'(\mathbb{T})^K)$ satisfies condition (S) if, for all $t \in [0,T]$, the $\mathcal{S}'(\mathbb{T})$-valued random variables $\{u_{k,t}\}_{k \in \Z_{K}}$ form a family of independent white noises of variance 1.

For a process $\{u_{t}: \ t \in [0,T] $, $0 \leq s \leq t \leq T \}$ satisfying condition (S), $\varphi \in \mathcal{S}(\T)$ and $\varepsilon > 0$, we define
\begin{equation*}
    \mathcal{A}_{s,t}^{\varepsilon,(i,j)}(\varphi) = 
    \begin{cases}
        \int_{s}^{t} \int_{\T} u_{i,v}(\overrightarrow{\iota_{\varepsilon}}(x)) u_{j,v}(\overrightarrow{\iota_{\varepsilon}}(x)) \partial_{x} \varphi(x) dxdv, \quad \text{if } {i \neq j}\\
        \int_{s}^{t} \int_{\T} u_{i,v}(\overleftarrow{\iota_{\varepsilon}}(x)) u_{i,v}(\overrightarrow{\iota_{\varepsilon}}(x)) \partial_{x} \varphi(x) dxdv, \quad \text{if } {i = j}\\
    \end{cases}
\end{equation*}
where $\overleftarrow{\iota_{\varepsilon}}(x) = \frac{1}{\varepsilon} \mathds{1}_{(x-\varepsilon,x]}$ and $\overrightarrow{\iota_{\varepsilon}}(x) = \frac{1}{\varepsilon} \mathds{1}_{(x,x + \varepsilon]}$.
\end{definition}
%%%%%%%%%%%%%%%%%%%%%%%%%%%%%%%%%%%%%%%%%%%%%%%%%%%%%%%%%%%%%%%
%%%%%%%%%%%%%%%%%%%%%%%%%%%%%%%%%%%%%%%%%%%%%%%%%%%%%%%%%%%%%%%

%%%%%%%%%%%%%%%%%%%%%%%%%%%%%%%%%%%%%%%%%%%%%%%%%%%%%%%%%%%%%%%
%%%%%%%%%%%%%%%%%%%%%%%%%%%%%%%%%%%%%%%%%%%%%%%%%%%%%%%%%%%%%%%
\begin{definition}
Let $\{u_{t}: \ t \in [0,T] \}$ be a process satisfying condition (S). We say that $\{u_{t}: \ t \in [0,T] \}$ satisfies the energy estimate (EC) if there exists a constant $\kappa > 0$ such that,
for any $\varphi \in \mathcal{S}(\T)$, any $0 \leq s \leq t \leq T$ and any $0 < \delta < \varepsilon < 1$,
    \begin{equation}
    \label{conditionEC}
        \E \left[ \Big| \mathcal{A}_{s,t}^{\varepsilon,(i,j)}(\varphi) - \mathcal{A}_{s,t}^{\delta,(i,j)}(\varphi) \Big|^{2} \right] \leq \kappa (t-s) \varepsilon \mathcal{E}(\partial_{x} \varphi),
    \end{equation}
for all $i,j \in \Z_{K}$.
\end{definition}
%%%%%%%%%%%%%%%%%%%%%%%%%%%%%%%%%%%%%%%%%%%%%%%%%%%%%%%%%%%%%%%
%%%%%%%%%%%%%%%%%%%%%%%%%%%%%%%%%%%%%%%%%%%%%%%%%%%%%%%%%%%%%%%
Conditions (S) and (EC) are the key to define the quadratic term. The following corresponds to \cite[Theorem 1]{goncalvesJaraEnergysol} in the multi-component context. The proof is identical.
%%%%%%%%%%%%%%%%%%%%%%%%%%%%%%%%%%%%%%%%%%%%%%%%%%%%%%%%%%%%%%%
%%%%%%%%%%%%%%%%%%%%%%%%%%%%%%%%%%%%%%%%%%%%%%%%%%%%%%%%%%%%%%%
\begin{theorem}
\label{ThmAconvergence}
Assume that $\{u_{t}: \ t \in [0,T] \}$ satisfies (S) and \eqref{conditionEC} for some $i,j \in \Z_{K}$. Then, there exists an $\mathcal{S}'(\T)$-valued stochastic process $\{\mathcal{A}_{t}: \ t \in [0,T] \}$ with continuous paths such that
\begin{equation*}
    \mathcal{A}_{t}^{(i,j)}(\varphi) = \lim_{\varepsilon \to 0} \mathcal{A}_{0,t}^{\varepsilon,(i,j)}(\varphi),
\end{equation*}
in $L^{2}$, for any $t \in [0,T]$ and $\varphi \in \mathcal{S}(\T)$.
\end{theorem}
%%%%%%%%%%%%%%%%%%%%%%%%%%%%%%%%%%%%%%%%%%%%%%%%%%%%%%%%%%%%%%%
%%%%%%%%%%%%%%%%%%%%%%%%%%%%%%%%%%%%%%%%%%%%%%%%%%%%%%%%%%%%%%%
In this way, all the components of the quadratic term are well-defined for processes satisfying (S) and (EC). We can now state the definition of energy solutions.
%%%%%%%%%%%%%%%%%%%%%%%%%%%%%%%%%%%%%%%%%%%%%%%%%%%%%%%%%%%%%%%
%%%%%%%%%%%%%%%%%%%%%%%%%%%%%%%%%%%%%%%%%%%%%%%%%%%%%%%%%%%%%%%
\begin{definition}
We say that $\{u_{t}: \ t \in [0,T] \}$ is a stationary energy solution of the coupled Burgers equations \eqref{coupledBurgers} if
\begin{itemize}
    \item $\{u_{t}: \ t \in [0,T] \}$ satisfies (S) and (EC).
    \pagebreak
    \item For all $\varphi \in \mathcal{S}(\T)$, for all $k \in \Z_{K}$, the process
    \begin{equation*}
        u_{k,t}(\varphi) - u_{k,0}(\varphi) - \frac{1}{2} \int_{0}^{t} u_{k,s}(\partial_{x}^{2} \varphi)ds - \mathcal{A}_{t}^{k}(\varphi),
    \end{equation*}
    is a martingale with quadratic variation $t \mathcal{E}(\partial_{x} \varphi)$, where
    \begin{equation*}
        \mathcal{A}_{t}^{k} = \sum_{i,j = 1}^{K} \Gamma_{i,j}^{k} \mathcal{A}_{t}^{(i,j)},
    \end{equation*}
    the processes $\mathcal{A}_{t}^{(i,j)}$ are given by Theorem \ref{ThmAconvergence}
    and $(\Gamma_{i,j}^{k})_{i,j,k}$ are the coefficients from the coupled stochastic Burgers equations \eqref{coupledBurgers}.
    \item The time-reversed process $\hat{u} = \{u_{T-t}: \ t \in [0,T] \}$ satisfies that, for each $k \in \Z_{K}$ and $\varphi \in \mathcal{S}(\T)$,
    \begin{equation*}
        \hat{u}_{k,t}(\varphi) - \hat{u}_{k,0}(\varphi) - \frac{1}{2} \int_{0}^{t} \hat{u}_{k,s}(\partial_{x}^{2} \varphi)ds + \hat{\mathcal{A}}_{t}^{k}(\varphi) 
    \end{equation*}
    is a martingale with quadratic variation $t \mathcal{E}(\partial_{x} \varphi)$ in the filtration generated by $\hat{u}$, where $\hat{\mathcal{A}}_{t} = \mathcal{A}_{T} - \mathcal{A}_{T-t}$.
\end{itemize}
\end{definition}
%%%%%%%%%%%%%%%%%%%%%%%%%%%%%%%%%%%%%%%%%%%%%%%%%%%%%%%%%%%%%%%
%%%%%%%%%%%%%%%%%%%%%%%%%%%%%%%%%%%%%%%%%%%%%%%%%%%%%%%%%%%%%%%
Existence and uniqueness of energy solutions are proved in \cite{GP-generator}.
%%%%%%%%%%%%%%%%%%%%%%%%%%%%%%%%%%%%%%%%%%%%%%%%%%%%%%%%%%%%%%%
%%%%%%%%%%%%%%%%%%%%%%%%%%%%%%%%%%%%%%%%%%%%%%%%%%%%%%%%%%%%%%%
\begin{remark}\label{rk:quadratic}
Our quadratic term $\mathcal{A}_{s,t}^{\varepsilon,(i,i)}$ is slightly different from the standard one, which is usually defined as
\begin{equation*}
    \Tilde{\mathcal{A}}_{s,t}^{\varepsilon,(i,i)} =  \int_{s}^{t} \int_{\T} u_{i,v}(\overrightarrow{\iota_{\varepsilon}}(x))^2  \partial_{x} \varphi(x) dxdv.
\end{equation*}
However, it can be showed that $ \displaystyle \lim_{\varepsilon \to 0} \Tilde{\mathcal{A}}_{t}^{\varepsilon,(i,i)} = \lim_{\varepsilon \to 0} \mathcal{A}_{t}^{\varepsilon,(i,i)}$. This can be seen, for instance, by combining our second-order Boltzmann-Gibbs principle Theorem \ref{thm:BoltzmannGibbs} with a version of the corresponding result in \cite{jaraMoreno} in the multi-component setting, and taking the weakly asymmetric limit. The required modifications of Theorem \ref{ThmAconvergence} are straightforward.
\end{remark}
%%%%%%%%%%%%%%%%%%%%%%%%%%%%%%%%%%%%%%%%%%%%%%%%%%%%%%%%%%%%%%%
%%%%%%%%%%%%%%%%%%%%%%%%%%%%%%%%%%%%%%%%%%%%%%%%%%%%%%%%%%%%%%%

%%%%%%%%%%%%%%%%%%%%%%%%%%%%%%%%%%%%%%%%%%%%%%%%%%%%%%%%%%%%%%%
%%%%%%%%%%%%%%%%%%%%%%%%%%%%%%%%%%%%%%%%%%%%%%%%%%%%%%%%%%%%%%%
%\begin{remark}
%The theory of energy solutions has been extremely successful to show convergence of stationary discrete models to the Burgers equation. See for instance \cite{goncalvesJaraEnergysol, GJS, DGP, jaraMoreno, JaM2}.
%\end{remark}
%%%%%%%%%%%%%%%%%%%%%%%%%%%%%%%%%%%%%%%%%%%%%%%%%%%%%%%%%%%%%%%
%%%%%%%%%%%%%%%%%%%%%%%%%%%%%%%%%%%%%%%%%%%%%%%%%%%%%%%%%%%%%%%

%%%%%%%%%%%%%%%%%%%%%%%%%%%%%%%%%%%%%%%%%%%%%%%%%%%%%%%%%%%%%%%
%%%%%%%%%%%%%%%%%%%%%%%%%%%%%%%%%%%%%%%%%%%%%%%%%%%%%%%%%%%%%%%
%%%%%%%%%%%%%%%%%%%%%%%%%%%%%%%%%%%%%%%%%%%%%%%%%%%%%%%%%%%%%%%
%%%%%%%%%%%%%%%%%%%%%%%%%%%%%%%%%%%%%%%%%%%%%%%%%%%%%%%%%%%%%%%
%%%%%%%%%%%%%%%%%%%%%%%%%%%%%%%%%%%%%%%%%%%%%%%%%%%%%%%%%%%%%%%
%%%%%%%%%%%%%%%%%%%%%%%%%%%%%%%%%%%%%%%%%%%%%%%%%%%%%%%%%%%%%%%
%%%%%%%%%%%%%%%%%%%%%%%%%%%%%%%%%%%%%%%%%%%%%%%%%%%%%%%%%%%%%%%
%%%%%%%%%%%%%%%%%%%%%%%%%%%%%%%%%%%%%%%%%%%%%%%%%%%%%%%%%%%%%%%

\section{Coupled Sasamoto-Spohn Processes}
\label{sec:coupled-SS}

%%%%%%%%%%%%%%%%%%%%%%%%%%%%%%%%%%%%%%%%%%%%%%%%%%%%%%%%%%%%%%%
%%%%%%%%%%%%%%%%%%%%%%%%%%%%%%%%%%%%%%%%%%%%%%%%%%%%%%%%%%%%%%%
%%%%%%%%%%%%%%%%%%%%%%%%%%%%%%%%%%%%%%%%%%%%%%%%%%%%%%%%%%%%%%%
%%%%%%%%%%%%%%%%%%%%%%%%%%%%%%%%%%%%%%%%%%%%%%%%%%%%%%%%%%%%%%%

\subsection{The Generator and the invariant measure}\label{sec:invariant-measure}

%%%%%%%%%%%%%%%%%%%%%%%%%%%%%%%%%%%%%%%%%%%%%%%%%%%%%%%%%%%%%%%
%%%%%%%%%%%%%%%%%%%%%%%%%%%%%%%%%%%%%%%%%%%%%%%%%%%%%%%%%%%%%%%
%%%%%%%%%%%%%%%%%%%%%%%%%%%%%%%%%%%%%%%%%%%%%%%%%%%%%%%%%%%%%%%
%%%%%%%%%%%%%%%%%%%%%%%%%%%%%%%%%%%%%%%%%%%%%%%%%%%%%%%%%%%%%%%
Let $\mathscr{C}$ be the space of twice continuously differentiable functions from $\R^{KM}$ to $\R$ with polynomial growth of their derivatives up to order two. The following lemma follows from a simple application of It\^o's formula. 
%%%%%%%%%%%%%%%%%%%%%%%%%%%%%%%%%%%%%%%%%%%%%%%%%%%%%%%%%%%%%%%
%%%%%%%%%%%%%%%%%%%%%%%%%%%%%%%%%%%%%%%%%%%%%%%%%%%%%%%%%%%%%%%
\begin{lemma}
The generator of the dynamics (\ref{eq:SS-propio}) acts on $\mathscr{C}$ as
\begin{equation}
    L = \sum_{k \in \Z_{K}} \sum_{j \in \Z_{M}} \Big\{ \frac{1}{2} (\partial_{k,j+1} - \partial_{k,j})^{2} -\frac{1}{2} (u_{k,j+1}-u_{k,j}) (\partial_{k,j+1} - \partial_{k,j}) \\ + \epsilon B_{k,j}(u) \partial_{k,j} \Big\}
\end{equation}
\end{lemma}
%%%%%%%%%%%%%%%%%%%%%%%%%%%%%%%%%%%%%%%%%%%%%%%%%%%%%%%%%%%%%%%
%%%%%%%%%%%%%%%%%%%%%%%%%%%%%%%%%%%%%%%%%%%%%%%%%%%%%%%%%%%%%%%

We recall that $\mu_{K,M} = (\rho^{\otimes \Z_{K}})^{\otimes \Z_{M}}$, where $d\rho(x) = \frac{1}{\sqrt{2 \pi}} e^{-\frac{x^{2}}{2}}dx$.
It is a standard fact that we have the integration-by-parts formula
\begin{eqnarray*}
	\int_{\R} x f(x) d\rho(x)
	=
	\int_{\R} \partial_{x} f(x) d\rho(x),
\end{eqnarray*}
for all $f\in C^1(\R)$, such that $f$ and $\partial_x f$ grow polynomially.
We also recall that the coefficients of the model \eqref{eq:SS-propio} are assumed to satisfy conditions \eqref{betaGamma}-\eqref{lambda}, namely,
\begin{align*}
    \beta_{k}^{a} = 2 \gamma_{k+a}^{a},
    \quad
    \lambda_{k}^{k-a,k-a'} = \lambda_{k}^{k-a',k-a} = \lambda_{k-a}^{k,k-a'}
\end{align*}
for all $k,a,a' \in \Z_{K}$.
%%%%%%%%%%%%%%%%%%%%%%%%%%%%%%%%%%%%%%%%%%%%%%%%%%%%%%%%%%%%%%%
%%%%%%%%%%%%%%%%%%%%%%%%%%%%%%%%%%%%%%%%%%%%%%%%%%%%%%%%%%%%%%%
\begin{lemma}
\label{adjointOfL}
Assume \eqref{betaGamma}-\eqref{lambda}.
The adjoint of the generator $L$ in $L^{2}(\mu_{K,M})$ is given by
\begin{equation*}
    L^{*} =  \sum_{k \in \Z_{K}} \sum_{j \in \Z_{M}} \Big\{ \frac{1}{2}(\partial_{k,j+1} - \partial_{k,j})^{2} - \frac{1}{2}(u_{k,j+1} - u_{k,j})(\partial_{k,j+1} - \partial_{k,j}) - \epsilon B_{k,j}(u) \partial_{k,j} \Big\}.
\end{equation*}
\end{lemma}
%%%%%%%%%%%%%%%%%%%%%%%%%%%%%%%%%%%%%%%%%%%%%%%%%%%%%%%%%%%%%%%
%%%%%%%%%%%%%%%%%%%%%%%%%%%%%%%%%%%%%%%%%%%%%%%%%%%%%%%%%%%%%%%
\begin{proof}
First, we compute the adjoint of $\partial_{k,j}$ in $L^{2}(\mu_{K,M})$. Let $f$ and $g$ be twice differentiable. Then,
\begin{align*}
    \int_{\R^{KM}} \partial_{k,j} f(u) g(u) d\mu_{K,M}(u) 
    & = \int_{\R^{KM}} \partial_{k,j}f(u) \Big\{ g(u)\rho_{K,M}(u)  \Big\} du \\
    & =  - \int_{\R^{KM}} f(u) \partial_{k,j} (g(u) \rho_{K,M}(u)) du \\
    & = \int_{\R^{KM}} f(u) \Big\{ (-\partial_{k,j} + u_{k,j}) g(u) \Big\} d\mu_{K,M}(u).
\end{align*}
We conclude that $\partial_{k,j}^{*} = -\partial_{k,j} + u_{k,j} $. Hence,
\begin{align*}
    ((\partial_{k,j+1} - \partial_{k,j})^{2})^{*} &= ((\partial_{k,j+1} - \partial_{k,j})^{*})^{2} \\
    & = (\partial_{k,j} - \partial_{k,j+1} + u_{k,j+1} - u_{k,j})^{2} \\
    & = (\partial_{k,j+1} - \partial_{k,j})^{2} + (\partial_{k,j} - \partial_{k,j+1})(u_{k,j+1} - u_{k,j}) \\
    & \quad - (u_{k,j+1} - u_{k,j})(\partial_{k,j+1} - \partial_{k,j}) + (u_{k,j+1} - u_{k,j})^{2}.
\end{align*}
On the other hand,
\begin{align*}
    (\partial_{k,j} - \partial_{k,j+1})((u_{k,j+1} - u_{k,j})f) & = -2f + (u_{k,j+1} - u_{k,j})(\partial_{k,j} - \partial_{k,j+1})f.
\end{align*}
Combining the last two displays, we conclude that
\begin{equation*}
    ((\partial_{k,j+1} - \partial_{k,j})^{2})^{*} = (\partial_{k,j+1} - \partial_{k,j})^{2} - 2(u_{k,j+1} - u_{k,j})(\partial_{k,j+1} - \partial_{k,j}) + (u_{k,j+1} - u_{k,j})^{2} - 2.
\end{equation*}
Next,
\begin{align*}
    ((u_{k,j+1} - u_{k,j})(\partial_{k,j+1} - \partial_{k,j}))^{*} & = (\partial_{k,j+1} - \partial_{k,j})^{*}(u_{k,j+1} - u_{k,j})^{*} \\
    & = (\partial_{k,j} - \partial_{k,j+1} + u_{k,j+1} - u_{k,j})(u_{k,j+1} - u_{k,j}) \\
    & = -2 + (u_{k,j+1} - u_{k,j})(\partial_{k,j} - \partial_{k,j+1}) + (u_{k,j+1} - u_{k,j})^{2}.
\end{align*}
Finally,
\begin{align*}
    (B_{k,j}(u) \partial_{k,j})^{*}f &= \partial_{k,j}^{*} (B_{k,j}(u))f \\
    & = (-\partial_{k,j} + u_{k,j}) B_{k,j}(u) f \\
    & = -(\partial_{k,j}B_{k,j}(u)) f - B_{k,j}(u) \partial_{k,j} f + u_{k,j} B_{k,j}(u) f.
\end{align*}
Putting the above computations together, we obtain
\begin{align*}
    L^{*} 
    & = \sum_{k \in \Z_{K}} \sum_{j \in \Z_{M}} \Big\{ \frac{1}{2}(\partial_{k,j+1} - \partial_{k,j})^{2} - \frac{1}{2}(u_{k,j+1} - u_{k,j})(\partial_{k,j+1} - \partial_{k,j}) - \epsilon B_{k,j}(u) \partial_{k,j}  \\
    & \phantom{blablablablabla}
    + \epsilon (-\partial_{k,j}B_{k,j}(u) + u_{k,j} B_{k,j}(u)) \Big\}.
\end{align*}
We conclude the proof by noticing that
\begin{equation*}
    \sum_{k \in \Z_{K}} \sum_{j \in \Z_{M}} \epsilon (-\partial_{k,j}B_{k,j}(u) + u_{k,j} B_{k,j}(u)) = 0,
\end{equation*}
which is the object of Lemma \ref{ujBj-djBj=0} below.
\end{proof}
%%%%%%%%%%%%%%%%%%%%%%%%%%%%%%%%%%%%%%%%%%%%%%%%%%%%%%%%%%%%%%%
%%%%%%%%%%%%%%%%%%%%%%%%%%%%%%%%%%%%%%%%%%%%%%%%%%%%%%%%%%%%%%%
The next lemma summarizes a key identity in the proof above and will be used once again in the proof of stationarity of the Gaussian distribution.
%%%%%%%%%%%%%%%%%%%%%%%%%%%%%%%%%%%%%%%%%%%%%%%%%%%%%%%%%%%%%%%
%%%%%%%%%%%%%%%%%%%%%%%%%%%%%%%%%%%%%%%%%%%%%%%%%%%%%%%%%%%%%%%
\begin{lemma}
\label{ujBj-djBj=0}
Let $B_{k,j}$ be defined as in Section \ref{ssmodel} and assume \eqref{betaGamma}-\eqref{lambda}. Then,
\begin{equation*}
    \sum_{k \in \Z_{K}} \sum_{j \in \Z_{M}} (u_{k,j} B_{k,j}(u) -\partial_{k,j}B_{k,j}(u)) = 0.
\end{equation*}
\end{lemma}
%%%%%%%%%%%%%%%%%%%%%%%%%%%%%%%%%%%%%%%%%%%%%%%%%%%%%%%%%%%%%%%
%%%%%%%%%%%%%%%%%%%%%%%%%%%%%%%%%%%%%%%%%%%%%%%%%%%%%%%%%%%%%%%
The proof is deferred to Appendix \ref{sec:summation-lemma}.

%%%%%%%%%%%%%%%%%%%%%%%%%%%%%%%%%%%%%%%%%%%%%%%%%%%%%%%%%%%%%%%
%%%%%%%%%%%%%%%%%%%%%%%%%%%%%%%%%%%%%%%%%%%%%%%%%%%%%%%%%%%%%%%
Now, we introduce the operators
\begin{eqnarray}
	\label{symmetricOfL}
	S 
	&=& 
	\frac{L + L^{*}}{2} = \sum_{k \in \Z_{K}} \sum_{j \in \Z_{M}} \frac{1}{2} (\partial_{k,j+1} - \partial_{k,j})^{2} -\frac{1}{2} (u_{k,j+1}-u_{k,j}) (\partial_{k,j+1} - \partial_{k,j}),
	\\
	\label{antisymmetricOfL}
	A 
	&=& 
	\frac{L - L^{*}}{2} = \sum_{k \in \Z_{K}} \sum_{j \in \Z_{M}} \epsilon B_{k,j}(u) \partial_{k,j}
\end{eqnarray}
which correspond to the symmetric and anti-symmetric parts of $L$ with respect to $\mu_{K,M}$. Their adjoints are simply given by
\begin{equation*}
	S^{*} = \frac{L^{*} + L}{2} = S,
	\quad
	A^{*} = \frac{L^{*} - L}{2} = -A.
\end{equation*}
%%%%%%%%%%%%%%%%%%%%%%%%%%%%%%%%%%%%%%%%%%%%%%%%%%%%%%%%%%%%%%%
%%%%%%%%%%%%%%%%%%%%%%%%%%%%%%%%%%%%%%%%%%%%%%%%%%%%%%%%%%%%%%%
We can now prove Proposition \ref{thm:invariant-distribution}. We formulate it once again for the convenience of the reader.
%%%%%%%%%%%%%%%%%%%%%%%%%%%%%%%%%%%%%%%%%%%%%%%%%%%%%%%%%%%%%%%
%%%%%%%%%%%%%%%%%%%%%%%%%%%%%%%%%%%%%%%%%%%%%%%%%%%%%%%%%%%%%%%
\begin{proposition}
\label{PropInvariantMeasure}
Assume \eqref{betaGamma}-\eqref{lambda}.
The measure $\mu_{K,M}$ is invariant for the dynamics (\ref{eq:SS-propio}).
\end{proposition}
%%%%%%%%%%%%%%%%%%%%%%%%%%%%%%%%%%%%%%%%%%%%%%%%%%%%%%%%%%%%%%%
%%%%%%%%%%%%%%%%%%%%%%%%%%%%%%%%%%%%%%%%%%%%%%%%%%%%%%%%%%%%%%%
\begin{proof}
The lemma follows from Echeverría's criterion (\cite{echeverria}, Theorem 4.9.17) once we show that
\begin{equation*}
    \int_{\R^{KM}} Lf(u) d\mu_{K,M}(u) = 0,
\end{equation*}
for all $f \in \mathscr{C}$. We will prove this for the symmetric and anti-symmetric parts of $L$ separately.
%%%%%%%%%%%%%%%%%%%%%%%%%%%%%%%%%%%%%%%%%%%%%%%%%%%%%%%%%%%%%%%
First,
\begin{eqnarray*}
	\int_{\R^{KM}} Sf(u) d\mu_{K,M}(u) 
	= \int_{\R^{KM}} Sf(u) \rho_{K,M}(u) du
    = \int_{\R^{KM}} f(u) S^{\dagger}\rho_{K,M}(u) du,
\end{eqnarray*}
where $S^{\dagger}$ is the adjoint of $S$ with respect to the Lebesgue measure on $\R^{KM}$. Using standard integration-by-parts, one can easily show that
\begin{equation*}
        S^{\dagger} = \frac{1}{2} \sum_{k \in \Z_{K}} \sum_{j \in \Z_{M}} (\partial_{k,j+1} - \partial_{k,j})^{2} + (u_{k,j+1}-u_{k,j}) (\partial_{k,j+1} - \partial_{k,j}) + 2.
\end{equation*}
A tedious but rather straightforward computation then shows that $S^{\dagger}\rho_{K,M}=0$. As a result,
\begin{equation*}
    \int_{\R^{KM}} Sf(u) d\mu_{K,M}(u) = 0,
\end{equation*}
for all $f \in \mathscr{C}$.

We now prove the corresponding identity for $A$. This time, it will not hold that $A^{\dagger} \rho_{K,M}=0$ in general and the argument relies heavily on the explicit structure of the non-linear term. Using standard integration-by-parts,
\begin{eqnarray*}
	\int_{\R^{KM}} Af d\mu_{K,M} 
	&=& 
	\epsilon
	\int_{\R^{KM}} \sum_{k \in \Z_{K}} \sum_{j \in \Z_{M}}  B_{k,j}(u) \partial_{k,j} f(u) d\mu_{K,M}(u) 
	\\
    	&=& 
    	\epsilon
    	\sum_{k \in \Z_{K}} \sum_{j \in \Z_{M}} \int_{\R^{KM}} B_{k,j}(u) \rho_{K,M}(u)  \partial_{k,j} f(u) du
    	 \\
    &=& 
    -\epsilon
    \sum_{k \in \Z_{K}} \sum_{j \in \Z_{M}}\int_{\R^{KM}} f(u) \partial_{k,j} \left\{ B_{k,j}(u) \rho_{K,M}(u) \right\} du
    \\
    &=&
    -\epsilon
    \sum_{k \in \Z_{K}} \sum_{j \in \Z_{M}}\int_{\R^{KM}} f(u) 
    		\left\{
    			\partial_{k,j}B_{k,j}(u) - u_{k,j}B_{k,j}(u)
    		\right\}
    	\rho_{K,M}(u) du.
\end{eqnarray*}
An application of Lemma \ref{ujBj-djBj=0} then shows that
\begin{eqnarray*}
	\int_{\R^{KM}} Af(u) d\mu_{K,M}(u) = 0,
\end{eqnarray*}
for all $f \in \mathscr{C}$.
\end{proof}
%%%%%%%%%%%%%%%%%%%%%%%%%%%%%%%%%%%%%%%%%%%%%%%%%%%%%%%%%%%%%%%
%%%%%%%%%%%%%%%%%%%%%%%%%%%%%%%%%%%%%%%%%%%%%%%%%%%%%%%%%%%%%%%

%%%%%%%%%%%%%%%%%%%%%%%%%%%%%%%%%%%%%%%%%%%%%%%%%%%%%%%%%%%%%%%
%%%%%%%%%%%%%%%%%%%%%%%%%%%%%%%%%%%%%%%%%%%%%%%%%%%%%%%%%%%%%%%
%%%%%%%%%%%%%%%%%%%%%%%%%%%%%%%%%%%%%%%%%%%%%%%%%%%%%%%%%%%%%%%
%%%%%%%%%%%%%%%%%%%%%%%%%%%%%%%%%%%%%%%%%%%%%%%%%%%%%%%%%%%%%%%

\subsection{The Martingale Decomposition}
\label{sec:MartingaleDescomposition}

%%%%%%%%%%%%%%%%%%%%%%%%%%%%%%%%%%%%%%%%%%%%%%%%%%%%%%%%%%%%%%%
%%%%%%%%%%%%%%%%%%%%%%%%%%%%%%%%%%%%%%%%%%%%%%%%%%%%%%%%%%%%%%%
%%%%%%%%%%%%%%%%%%%%%%%%%%%%%%%%%%%%%%%%%%%%%%%%%%%%%%%%%%%%%%%
%%%%%%%%%%%%%%%%%%%%%%%%%%%%%%%%%%%%%%%%%%%%%%%%%%%%%%%%%%%%%%%

Let $\varphi \in \mathcal{S}(\T)$ be a test function. Remember that the fluctuation field is given by
\begin{equation*}
    \mathcal{X}_{k,t}^{n}(\varphi) = \frac{1}{n^{\frac{1}{2}}} \sum_{j \in \Z_{n}} u_{k,j}(tn^2) \varphi_{j}^{n}, \quad k \in \Z_{K},
\end{equation*}
where it is understood that $u$ corresponds to the solution of \eqref{eq:SS-propio} with $\epsilon=n^{-\frac{1}{2}}$.
We decompose $\mathcal{X}_{k,t}^{n}(\varphi)$ into  its symmetric, anti-symmetric and martingale parts defined as
\begin{eqnarray*}
	\mathcal{S}_{k,t}^{n}(\varphi) 
	&=& 
	\int^t_0
	n^2 S \mathcal{X}_{k,s}^{n}(\varphi)ds
	=
	\int^t_0
	\frac{1}{2 n^{\frac{1}{2}}} \sum_{j \in \Z_{n}} u_{k,j}(sn^2) \Delta^{n} \varphi_{j}^{n}ds,
	\\
	\mathcal{A}_{k,t}^{n}(\varphi) 
	&=& 
	\int^t_0
	n^2 A \mathcal{X}_{k,s}^{n}(\varphi)ds
	=
	-\int^t_0
	\sum_{j \in \Z_{n}} G_{k,j}(sn^2) \nabla^{n} \varphi_{j}^{n} ds,
	\\
	\mathcal{M}_{k,t}^{n}(\varphi)
	&=&
	\mathcal{X}_{k,t}^{n}(\varphi)
	-\mathcal{X}_{k,0}^{n}(\varphi)
	-\mathcal{S}_{k,t}^{n}(\varphi)
	-\mathcal{A}_{k,t}^{n}(\varphi),
\end{eqnarray*}
respectively. Note that the martingale part of the dynamics is explicitly given by
\begin{equation*}
    \mathcal{M}_{k,t}^{n}(\varphi) = \frac{1}{n^{\frac{1}{2}}} \int_{0}^{tn^2} \sum_{j \in \Z_{n}} (\varphi_{j}^{n} - \varphi_{j+1}^{n}) d\xi_{k,j}(s),
\end{equation*}
and has quadratic variation
\begin{equation*}
    \langle \mathcal{M}_{k,\cdot}^{n} (\varphi) \rangle_{t} = t \mathcal{E}_{n} (\nabla^{n} \varphi^{n}),
    \quad
    \text{where}
    \quad
    \mathcal{E}_{n} (\nabla^{n} \varphi^{n})
    =
    \frac{1}{n} \sum_{j\in \Z_n} (\nabla^{n} \varphi^{n})^2.
\end{equation*}

%%%%%%%%%%%%%%%%%%%%%%%%%%%%%%%%%%%%%%%%%%%%%%%%%%%%%%%%%%%%%%%
%%%%%%%%%%%%%%%%%%%%%%%%%%%%%%%%%%%%%%%%%%%%%%%%%%%%%%%%%%%%%%%
%%%%%%%%%%%%%%%%%%%%%%%%%%%%%%%%%%%%%%%%%%%%%%%%%%%%%%%%%%%%%%%
%%%%%%%%%%%%%%%%%%%%%%%%%%%%%%%%%%%%%%%%%%%%%%%%%%%%%%%%%%%%%%%
%%%%%%%%%%%%%%%%%%%%%%%%%%%%%%%%%%%%%%%%%%%%%%%%%%%%%%%%%%%%%%%
%%%%%%%%%%%%%%%%%%%%%%%%%%%%%%%%%%%%%%%%%%%%%%%%%%%%%%%%%%%%%%%
%%%%%%%%%%%%%%%%%%%%%%%%%%%%%%%%%%%%%%%%%%%%%%%%%%%%%%%%%%%%%%%
%%%%%%%%%%%%%%%%%%%%%%%%%%%%%%%%%%%%%%%%%%%%%%%%%%%%%%%%%%%%%%%

\section{Dynamical Estimates}
\label{sec:Estimates}

%%%%%%%%%%%%%%%%%%%%%%%%%%%%%%%%%%%%%%%%%%%%%%%%%%%%%%%%%%%%%%%
%%%%%%%%%%%%%%%%%%%%%%%%%%%%%%%%%%%%%%%%%%%%%%%%%%%%%%%%%%%%%%%
%%%%%%%%%%%%%%%%%%%%%%%%%%%%%%%%%%%%%%%%%%%%%%%%%%%%%%%%%%%%%%%
%%%%%%%%%%%%%%%%%%%%%%%%%%%%%%%%%%%%%%%%%%%%%%%%%%%%%%%%%%%%%%%

The goal of this section is to prove the second order Boltzmann-Gibbs principle (Proposition \ref{thm:BoltzmannGibbs}), which is our main technical estimate. This will be a consequence of the one-block estimates Lemma \ref{thm:one-block-forward} and \ref{thm:one-block-backward}. We start by recalling some general results.

%%%%%%%%%%%%%%%%%%%%%%%%%%%%%%%%%%%%%%%%%%%%%%%%%%%%%%%%%%%%%%%
%%%%%%%%%%%%%%%%%%%%%%%%%%%%%%%%%%%%%%%%%%%%%%%%%%%%%%%%%%%%%%%
%%%%%%%%%%%%%%%%%%%%%%%%%%%%%%%%%%%%%%%%%%%%%%%%%%%%%%%%%%%%%%%
%%%%%%%%%%%%%%%%%%%%%%%%%%%%%%%%%%%%%%%%%%%%%%%%%%%%%%%%%%%%%%%

\subsection{The Kipnis-Varadhan estimate}
\label{kipnis-varadhan}

%%%%%%%%%%%%%%%%%%%%%%%%%%%%%%%%%%%%%%%%%%%%%%%%%%%%%%%%%%%%%%%
%%%%%%%%%%%%%%%%%%%%%%%%%%%%%%%%%%%%%%%%%%%%%%%%%%%%%%%%%%%%%%%
%%%%%%%%%%%%%%%%%%%%%%%%%%%%%%%%%%%%%%%%%%%%%%%%%%%%%%%%%%%%%%%
%%%%%%%%%%%%%%%%%%%%%%%%%%%%%%%%%%%%%%%%%%%%%%%%%%%%%%%%%%%%%%%

We recall the Kipnis-Varadhan inequality in our context: there exists $C > 0$ such that
\begin{equation*}
    \E_{n} \left[ \sup_{0 \leq t \leq T} \left| \int_{0}^{t} F(u(sn^2))ds \right|^{2} \right] \leq CT\norm{F(\cdot)}_{-1,n}^{2},
\end{equation*}
where the $\norm{\cdot}_{-1,n}$-norm is defined through the variational formula
\begin{equation*}
    \norm{F}_{-1,n} = \sup_{f \in \mathscr{C}} \Bigg\{ 2 \int_{\R^{KM}} F(u) f(u) d\mu_{K,M}(u)  + n^2 \int_{\R^{KM}} f(u)Lf(u) d\mu_{K,M}(u) \Bigg\},
\end{equation*}
where we recall that $\mathscr{C}$ denotes the space of twice continuously differentiable functions from $\R^{KM}$ to $\R$ with polynomial growth of their derivatives up to order two.
The proof of this estimate in our context can be obtained by a straightforward adaptation of \cite[Corollary 3.5]{GP-uniqueness}.
Now, using the definition of $S$ and $A$, we can see that
\begin{eqnarray*}
    \int_{\R^{KM}} f(u)Lf(u) d\mu_{K,M} 
    &=& 
    \int_{\R^{KM}} f(u)Sf(u) d\mu_{K,M}.
\end{eqnarray*}
Using Gaussian integration-by-parts, one can show that the above is equal to
\begin{eqnarray*}
    -\frac{1}{2} \sum_{k \in \Z_{K}} \sum_{j \in \Z_{M}} \int_{\R^{KM}} ((\partial_{k,j+1}-\partial_{k,j}) f(u))^{2} d\mu_{K,M}(u)
    =:
    -\frac{1}{2} D_{M,K}(f),
\end{eqnarray*}
where $D_{M,K}(f)$ corrresponds to the Dirichlet form of $L$. This way,
\begin{eqnarray*}
	\norm{F}_{-1,n} 
	= 
	\sup_{f \in \mathscr{C}} \Bigg\{ 2 \int_{\R^{KM}} F(u) f(u) d\mu_{K,M}(u)  -\frac{n^2}{2} D_{M,K}(f) \Bigg\}.
\end{eqnarray*}

%%%%%%%%%%%%%%%%%%%%%%%%%%%%%%%%%%%%%%%%%%%%%%%%%%%%%%%%%%%%%%%
%%%%%%%%%%%%%%%%%%%%%%%%%%%%%%%%%%%%%%%%%%%%%%%%%%%%%%%%%%%%%%%
%%%%%%%%%%%%%%%%%%%%%%%%%%%%%%%%%%%%%%%%%%%%%%%%%%%%%%%%%%%%%%%
%%%%%%%%%%%%%%%%%%%%%%%%%%%%%%%%%%%%%%%%%%%%%%%%%%%%%%%%%%%%%%%

\subsection{The one-block estimates}

%%%%%%%%%%%%%%%%%%%%%%%%%%%%%%%%%%%%%%%%%%%%%%%%%%%%%%%%%%%%%%%
%%%%%%%%%%%%%%%%%%%%%%%%%%%%%%%%%%%%%%%%%%%%%%%%%%%%%%%%%%%%%%%
%%%%%%%%%%%%%%%%%%%%%%%%%%%%%%%%%%%%%%%%%%%%%%%%%%%%%%%%%%%%%%%
%%%%%%%%%%%%%%%%%%%%%%%%%%%%%%%%%%%%%%%%%%%%%%%%%%%%%%%%%%%%%%%
We prove the key estimates in the proof of the second-order Boltzmann-Gibbs principle. We define
\begin{eqnarray*}
	\overrightarrow{u}_{k,j}^{l} = \frac{1}{l}\sum_{q=1}^{l} u_{k,j+q},
	\quad
	\overleftarrow{u}_{k,j}^{l} = \frac{1}{l}\sum_{q=0}^{l-1} u_{k,j-q}.
\end{eqnarray*}
We also define the canonical shift $\tau_{k,j}u_{\bar{k},\bar{j}} = u_{k + \bar{k},j+\bar{j}}$ acting on functions as $\tau_{k,j}f(u)=f(\tau_{k,j} u)$.
%%%%%%%%%%%%%%%%%%%%%%%%%%%%%%%%%%%%%%%%%%%%%%%%%%%%%%%%%%%%%%%
%%%%%%%%%%%%%%%%%%%%%%%%%%%%%%%%%%%%%%%%%%%%%%%%%%%%%%%%%%%%%%%
\begin{lemma}[One-block estimate - forward version]
\label{thm:one-block-forward}
%Denote by $\tau_{k,j}$ the canonical shift $\tau_{k,j}u_{\bar{k},\bar{j}} = u_{k + \bar{k},j+\bar{j}}$ and let $\overrightarrow{u}_{k,j}^{l} = \frac{1}{l}\sum_{q=1}^{l} u_{k,j+q} $.\\
Let $1 \leq l \leq \frac{K}{2}$ and let $g: \R^{KM} \to \R$ be a function with zero mean with respect to $\mu_{K,M}$ such that $\supp{(g)}$ does not intersect $\{(0,1),\dots,(0,l)\}$. Let $g_{k,j} = g(\tau_{k,j}u)$. Then, there exists a constant $C > 0$ such that, for all $\varphi \in l^{2}(\Z_{M})$,
\begin{equation*}
    \E_{n} \Bigg[ \Bigg| \int_{0}^{t} ds \sum_{j \in \Z_{M}} g_{k,j}(sn^2) [u_{k,j+1}(sn^2) - \overrightarrow{u}_{k,j}^{l}(sn^2) ] \varphi_{j}  \Bigg|^{2} \Bigg] \leq C \frac{tl d}{n} \norm{g}^{2}_{L^{2}(\mu_{K,M})} \mathcal{E}_{n}(\varphi),
\end{equation*}
where $d$ denotes the diameter of the support of $g$.
\end{lemma}
%%%%%%%%%%%%%%%%%%%%%%%%%%%%%%%%%%%%%%%%%%%%%%%%%%%%%%%%%%%%%%%
%%%%%%%%%%%%%%%%%%%%%%%%%%%%%%%%%%%%%%%%%%%%%%%%%%%%%%%%%%%%%%%
\begin{proof}
Let $\displaystyle \psi_{i} = \frac{l-i}{l}$, $i = 0,\dots,l-1$. Then,
\begin{align*}
    u_{k,j+1} - \overrightarrow{u}_{k,j}^{l} 
    & = \frac{1}{l} \sum_{q=1}^{l} ( u_{k,j+1} - u_{k,j+q}) \\
    & = \frac{1}{l} \sum_{q=1}^{l} \sum_{i=1}^{q-1} (u_{k,j+i} - u_{k,j+i+1}) 
    = \frac{1}{l} \sum_{i=1}^{l-1} \sum_{q=i+1}^{l} (u_{k,j+i} - u_{k,j+i+1}) \\
    & = \sum_{i=1}^{l-1} \psi_{i} (u_{k,j+i} - u_{k,j+i+1}).
\end{align*}
Hence, writing $p : = j+i$,
\begin{align*}
     \sum_{j \in \Z_{M}} g_{k,j} \varphi_{j} (u_{k,j+1} - \overrightarrow{u}_{k,j}^{l} ) & = \sum_{j \in \Z_{M}} g_{k,j} \varphi_{j} \sum_{i=1}^{l-1} \psi_{i} (u_{k,j+i} - u_{k,j+i+1}) \\
     & = \sum_{p} \Bigg( \sum_{i=1}^{l-1} \psi_{i} \varphi_{p-i} g_{k,p-i} \Bigg) (u_{k,p} - u_{k,p+1}) \\
     & =: \sum_{p} F_{p}(u_{k,p} - u_{k,p+1}).
\end{align*}
Now, for $f \in \mathscr{C}$, using Gaussian integration-by-parts,
\begin{align*}
    2 \int_{\R^{KM}} \sum_{j \in \Z_{M}} g_{k,j} \varphi_{j} (u_{k,j+1} - \overrightarrow{u}_{k,j}^{l}) f d\mu_{K,M} & = 2 \int_{\R^{KM}} \sum_{p} F_{p} (u_{k,p} - u_{k,p+1}) f d\mu_{K,M} \\
    & = 2 \int_{\R^{KM}} \sum_{p} F_{p} (\partial_{k,p} - \partial_{k,p+1}) f d\mu_{K,M} \\
    & \leq \int_{\R^{KM}} \sum_{p} \Bigg\{ \alpha F_{p}^{2} + \frac{1}{\alpha}((\partial_{k,p} - \partial_{k,p+1})f)^{2} \Bigg\}d\mu_{K,M},
\end{align*}
by the weighted Young's inequality. Taking $\alpha = \frac{2}{n^2}$, we find that the above is bounded by
\begin{equation*}
    \frac{2}{n^2} \sum_{p} \int_{\R^{KM}} F_{p}^{2} d\mu_{K,M} + \frac{n^2}{2} \sum_{p} \int_{\R^{KM}} ((\partial_{k,p} - \partial_{k,p+1})f)^{2} d\mu_{K,M}.
\end{equation*}
By the Kipnis-Varadhan inequality, there exists $C > 0$ such that
\begin{equation*}
    \E_{n} \Bigg[ \Bigg| \int_{0}^{t} ds \sum_{j \in \Z_{M}} g_{k,j}(sn) [u_{k,j+1}(sn) - \overrightarrow{u}_{k,j}^{l}(sn) ] \varphi_{j}  \Bigg|^{2} \Bigg]  \leq \frac{C t}{n^2} \sum_{p} \int_{\R^{KM}} F_{p}^{2} d\mu_{K,M}.
\end{equation*}
We are then left with estimating the right-hand side above:
\begin{equation*}
    \sum_{p} \int_{\R^{KM}} F_{p}^{2} d\mu_{K,M} = \sum_{p} \int_{\R^{KM}} \Bigg( \sum_{i=1}^{l-1} \psi_{i} \varphi_{p-i} g_{k,p-i} \Bigg)^{2} d\mu_{K,M}.
\end{equation*}
Observe that, as $\supp{(g)} \cap \{(0,1),\dots, (0,l)\} = \emptyset$, it holds that, if $d$ denotes the diameter of $\supp{(g)}$ and if $\abs{j-j'} > d$, then $g_{k,j}$ and $g_{k,j'}$ are independent. 
Assume first that $d<l$.
Denoting $a_{i} =\psi_{i} \varphi_{p-i} g_{k,p-i}$, decomposing $l-1 = md + r$ with $0 \leq r < d$ and using Jensen's inequality, we have
\begin{align*}
    \E_{n} \left[ F_{p}^{2} \right] 
    & = \E_{n} \Bigg[ \Bigg( \sum_{z=0}^{d-1} \sum_{z'=0}^{m-1} a_{zm+z'+1} + \sum_{i=md+1}^{l-1} a_{i}  \Bigg)^{2} \Bigg] 
    \\
    & \leq 
    2 \E_{n} \Bigg[ \Bigg( \sum_{z=0}^{d-1} \sum_{z'=0}^{m-1} a_{zm+z'+1} \Bigg)^{2} \Bigg] + 2 \E_{n} \Bigg[ \Bigg( \sum_{i=md+1}^{l-1} a_{i} \Bigg)^{2} \Bigg] 
    \\
    & \leq 2d \sum_{z=0}^{d-1} \E_{n}\Bigg[ \Bigg( \sum_{z'=0}^{m-1} a_{zm+z'+1} \Bigg)^{2} \Bigg] + 2r\E_{n} \Bigg[ \sum_{i=md+1}^{l-1} a_{i}^{2} \Bigg] 
    \\
    & = 2d \sum_{z=0}^{d-1} \E_{n} \Bigg[  \sum_{z'=0}^{m-1} a_{zm+z'+1}^{2} \Bigg] + 2r \E_{n} \Bigg[ \sum_{i=md+1}^{l-1} a_{i}^{2} \Bigg]
    \\
    & \leq 2d \E_{n} \Bigg[  \sum_{i=1}^{l-1} a_{i}^{2} \Bigg] 
    = 2d \E_{n} \Bigg[ \sum_{i=1}^{l-1} \psi_{i}^{2} \varphi_{p-i}^{2} g_{k,p-i}^{2} \Bigg] 
    \\
    & \leq 2d \sum_{i=1}^{l-1} \varphi_{p-i}^{2} \E_{n}[g_{k,p-i}^{2}] 
    = 2d \sum_{i=1}^{l-1} \varphi_{p-i}^{2} \norm{g}_{L^{2}(\mu_{K,M})}^{2},
\end{align*}
where we used that $\psi_{i} \in [0,1]$ in the last inequality. If $d \geq l$, we use the crude bound $F_{p}^{2} \leq 2d \sum_{i=0}^{l-1} \varphi_{p-i}^{2} g_{k,p-i}^{2}$. In any case,
\begin{align*}
    \frac{t}{n^2} \sum_{p} \int_{\R^{KM}} F_{p}^{2} d\mu_{K,M} 
    & \leq \frac{2td}{n^2} \sum_{p}\sum_{i=0}^{l-1} \varphi_{p-i}^{2} \norm{g}_{L^{2}(\mu_{K,M})}^{2} \\
    & \leq \frac{2tld}{n^2} \sum_{j} \varphi_{j}^{2} \norm{g}_{L^{2}(\mu_{K,M})}^{2} \\
    & \leq \frac{2tld}{n} \norm{g}_{L^{2}(\mu_{K,M})}^{2} \mathcal{E}_{n}(\varphi).
\end{align*}
This finishes the proof.
\end{proof}
%%%%%%%%%%%%%%%%%%%%%%%%%%%%%%%%%%%%%%%%%%%%%%%%%%%%%%%%%%%%%%%
%%%%%%%%%%%%%%%%%%%%%%%%%%%%%%%%%%%%%%%%%%%%%%%%%%%%%%%%%%%%%%%

%%%%%%%%%%%%%%%%%%%%%%%%%%%%%%%%%%%%%%%%%%%%%%%%%%%%%%%%%%%%%%%
%%%%%%%%%%%%%%%%%%%%%%%%%%%%%%%%%%%%%%%%%%%%%%%%%%%%%%%%%%%%%%%
\begin{lemma}[One-block estimate - backward version]
\label{thm:one-block-backward}
%Let $\overleftarrow{u}_{k,j}^{l} = \frac{1}{l}\sum_{q=0}^{l-1} u_{k,j-q}$. 
Let  $1 \leq l \leq \frac{K}{2}$and let $g:\R^{KM} \to \R$ be a function with zero mean with respect to $\mu$ such that $\supp{(g)}$ does not intersect $\{(0,-1),\dots,(0,-l)\}$. Let $g_{k,j} = g(\tau_{k,j}u)$. Then, there exists a constant $C > 0$ such that, for all $\varphi \in l^{2}(\Z_{M})$,
\begin{equation*}
    \E_{n} \Bigg[ \Bigg| \int_{0}^{t} ds \sum_{j \in \Z_{M}} g_{k,j}(sn^2) [u_{k,j}(sn^2) - \overleftarrow{u}_{k,j}^{l}(sn^2) ] \varphi_{j}  \Bigg|^{2} \Bigg] \leq C \frac{tl d}{n} \norm{g}^{2}_{L^{2}(\mu_{K,M})} \mathcal{E}_{n}(\varphi).
\end{equation*}
\end{lemma}
%%%%%%%%%%%%%%%%%%%%%%%%%%%%%%%%%%%%%%%%%%%%%%%%%%%%%%%%%%%%%%%
%%%%%%%%%%%%%%%%%%%%%%%%%%%%%%%%%%%%%%%%%%%%%%%%%%%%%%%%%%%%%%%
\begin{proof}
The proof is completely analogous to the proof of Lemma \ref{thm:one-block-forward}.
\end{proof}
%%%%%%%%%%%%%%%%%%%%%%%%%%%%%%%%%%%%%%%%%%%%%%%%%%%%%%%%%%%%%%%
%%%%%%%%%%%%%%%%%%%%%%%%%%%%%%%%%%%%%%%%%%%%%%%%%%%%%%%%%%%%%%%

%%%%%%%%%%%%%%%%%%%%%%%%%%%%%%%%%%%%%%%%%%%%%%%%%%%%%%%%%%%%%%%
%%%%%%%%%%%%%%%%%%%%%%%%%%%%%%%%%%%%%%%%%%%%%%%%%%%%%%%%%%%%%%%
%%%%%%%%%%%%%%%%%%%%%%%%%%%%%%%%%%%%%%%%%%%%%%%%%%%%%%%%%%%%%%%
%%%%%%%%%%%%%%%%%%%%%%%%%%%%%%%%%%%%%%%%%%%%%%%%%%%%%%%%%%%%%%%

\subsection{The second order Boltzmann-Gibbs principle}
\label{subsec:boltzmann-gibbs}

%%%%%%%%%%%%%%%%%%%%%%%%%%%%%%%%%%%%%%%%%%%%%%%%%%%%%%%%%%%%%%%
%%%%%%%%%%%%%%%%%%%%%%%%%%%%%%%%%%%%%%%%%%%%%%%%%%%%%%%%%%%%%%%
%%%%%%%%%%%%%%%%%%%%%%%%%%%%%%%%%%%%%%%%%%%%%%%%%%%%%%%%%%%%%%%
%%%%%%%%%%%%%%%%%%%%%%%%%%%%%%%%%%%%%%%%%%%%%%%%%%%%%%%%%%%%%%%
We now prove our main estimates.
%%%%%%%%%%%%%%%%%%%%%%%%%%%%%%%%%%%%%%%%%%%%%%%%%%%%%%%%%%%%%%%
%%%%%%%%%%%%%%%%%%%%%%%%%%%%%%%%%%%%%%%%%%%%%%%%%%%%%%%%%%%%%%%
\begin{proposition}[Second-order Boltzmann-Gibbs principle]
\label{thm:BoltzmannGibbs}
Let  $1 \leq l \leq \frac{K}{2}$. There exists a constant $C > 0$ such that, for all $\varphi \in l^{2}(\Z_{M})$,
\begin{equation*}
    \E_{n} \Bigg[ \Bigg| \int_{0}^{t} ds \sum_{j \in \Z_{M}} [u_{k,j}(sn) u_{k,j+1}(sn) -\overleftarrow{u}_{k,j}^{l}(sn) \overrightarrow{u}_{k,j}^{l}(sn) ] \varphi_{j}  \Bigg|^{2} \Bigg] \leq C \frac{tl}{n} \mathcal{E}_{n}(\varphi).
\end{equation*}
\end{proposition}
%%%%%%%%%%%%%%%%%%%%%%%%%%%%%%%%%%%%%%%%%%%%%%%%%%%%%%%%%%%%%%%
%%%%%%%%%%%%%%%%%%%%%%%%%%%%%%%%%%%%%%%%%%%%%%%%%%%%%%%%%%%%%%%
\begin{proof}
We use the factorization
\begin{equation*}
    u_{k,j}u_{k,j+1} - \overleftarrow{u}_{k,j}^{l} \overrightarrow{u}_{k,j}^{l} = u_{k,j}(u_{k,j+1} - \overrightarrow{u}_{k,j}^{l}) + \overrightarrow{u}_{k,j}^{l}(u_{k,j} - \overleftarrow{u}_{k,j}^{l}).
\end{equation*}
We handle the first term with Lemma \ref{thm:one-block-forward} with $g_{k,j} = u_{k,j}$ and the second one with Lemma \ref{thm:one-block-backward} with $g_{k,j} = \overrightarrow{u}_{k,j}^{l}$, noting that the diameter of the support of $u_{k,j}$ and $\overrightarrow{u}_{k,j}^{l}$ is $1$ and $l$ respectively, and noting that
\begin{eqnarray*}
	\norm{u_{k,j}}^{2}_{L^{2}(\mu_{K,M})}=1,
	\quad
	\norm{\overrightarrow{u}_{k,j}^{l}}^{2}_{L^{2}(\mu_{K,M})} = \frac{1}{l}.
\end{eqnarray*}
\end{proof}
%%%%%%%%%%%%%%%%%%%%%%%%%%%%%%%%%%%%%%%%%%%%%%%%%%%%%%%%%%%%%%%
%%%%%%%%%%%%%%%%%%%%%%%%%%%%%%%%%%%%%%%%%%%%%%%%%%%%%%%%%%%%%%%

%%%%%%%%%%%%%%%%%%%%%%%%%%%%%%%%%%%%%%%%%%%%%%%%%%%%%%%%%%%%%%%
%%%%%%%%%%%%%%%%%%%%%%%%%%%%%%%%%%%%%%%%%%%%%%%%%%%%%%%%%%%%%%%
\begin{proposition}[Second-order Boltzmann-Gibbs principle for crossed terms]
\label{thm:BoltzmannGibbsCrossed}
Let $l \geq 1$. There exists a constant $C > 0$ such that, for all $\varphi \in l^{2}(\Z_{M})$ and $k \neq \bar{k}$,
\begin{equation*}
    \E_{n} \Bigg[ \Bigg| \int_{0}^{t} ds \sum_{j \in \Z_{M}} [u_{k,j}(sn) u_{\bar{k},j}(sn) -\overrightarrow{u}_{k,j-1}^{l}(sn) \overrightarrow{u}_{\bar{k},j-1}^{l}(sn) ] \varphi_{j}  \Bigg|^{2} \Bigg] \leq C \frac{tl}{n} \mathcal{E}_{n}(\varphi).
\end{equation*}
\end{proposition}
%%%%%%%%%%%%%%%%%%%%%%%%%%%%%%%%%%%%%%%%%%%%%%%%%%%%%%%%%%%%%%%
%%%%%%%%%%%%%%%%%%%%%%%%%%%%%%%%%%%%%%%%%%%%%%%%%%%%%%%%%%%%%%%
\begin{proof}
This time, we use the factorization
\begin{equation*}
    u_{k,j}u_{\bar{k},j} - \overrightarrow{u}_{k,j-1}^{l} \overrightarrow{u}_{\bar{k},j-1}^{l} = u_{k,j}(u_{\bar{k},j} - \overrightarrow{u}_{\bar{k},j-1}^{l}) + \overrightarrow{u}_{\bar{k},j-1}^{l}(u_{k,j} - \overrightarrow{u}_{k,j-1}^{l})
\end{equation*}
and proceed as in the proof of Proposition \ref{thm:BoltzmannGibbs}.
\end{proof}
%%%%%%%%%%%%%%%%%%%%%%%%%%%%%%%%%%%%%%%%%%%%%%%%%%%%%%%%%%%%%%%
%%%%%%%%%%%%%%%%%%%%%%%%%%%%%%%%%%%%%%%%%%%%%%%%%%%%%%%%%%%%%%%

%%%%%%%%%%%%%%%%%%%%%%%%%%%%%%%%%%%%%%%%%%%%%%%%%%%%%%%%%%%%%%%
%%%%%%%%%%%%%%%%%%%%%%%%%%%%%%%%%%%%%%%%%%%%%%%%%%%%%%%%%%%%%%%
%%%%%%%%%%%%%%%%%%%%%%%%%%%%%%%%%%%%%%%%%%%%%%%%%%%%%%%%%%%%%%%
%%%%%%%%%%%%%%%%%%%%%%%%%%%%%%%%%%%%%%%%%%%%%%%%%%%%%%%%%%%%%%%
%%%%%%%%%%%%%%%%%%%%%%%%%%%%%%%%%%%%%%%%%%%%%%%%%%%%%%%%%%%%%%%
%%%%%%%%%%%%%%%%%%%%%%%%%%%%%%%%%%%%%%%%%%%%%%%%%%%%%%%%%%%%%%%
%%%%%%%%%%%%%%%%%%%%%%%%%%%%%%%%%%%%%%%%%%%%%%%%%%%%%%%%%%%%%%%
%%%%%%%%%%%%%%%%%%%%%%%%%%%%%%%%%%%%%%%%%%%%%%%%%%%%%%%%%%%%%%%

\section{Tightness}
\label{sec:tightness}

%%%%%%%%%%%%%%%%%%%%%%%%%%%%%%%%%%%%%%%%%%%%%%%%%%%%%%%%%%%%%%%
%%%%%%%%%%%%%%%%%%%%%%%%%%%%%%%%%%%%%%%%%%%%%%%%%%%%%%%%%%%%%%%
%%%%%%%%%%%%%%%%%%%%%%%%%%%%%%%%%%%%%%%%%%%%%%%%%%%%%%%%%%%%%%%
%%%%%%%%%%%%%%%%%%%%%%%%%%%%%%%%%%%%%%%%%%%%%%%%%%%%%%%%%%%%%%%

In the following, we will use Mitoma’s criterion \cite{mitoma}: a sequence of random distributions $(\gamma^{n})_n$ is tight in
$C([0, T], \mathcal{S}'(\T))$ if and only if $\gamma^{n}(\varphi)$ is tight in $C([0, T], \R)$ for all $\varphi \in \mathcal{S}(\T)$. We will show tightness of the symmetric, anti-symetric and martingale parts separately. We fix $\varphi \in \mathcal{S}(\T)$ for the remainder of this section.

%%%%%%%%%%%%%%%%%%%%%%%%%%%%%%%%%%%%%%%%%%%%%%%%%%%%%%%%%%%%%%%
%%%%%%%%%%%%%%%%%%%%%%%%%%%%%%%%%%%%%%%%%%%%%%%%%%%%%%%%%%%%%%%
%%%%%%%%%%%%%%%%%%%%%%%%%%%%%%%%%%%%%%%%%%%%%%%%%%%%%%%%%%%%%%%
%%%%%%%%%%%%%%%%%%%%%%%%%%%%%%%%%%%%%%%%%%%%%%%%%%%%%%%%%%%%%%%

\subsection{Martingale term}

%%%%%%%%%%%%%%%%%%%%%%%%%%%%%%%%%%%%%%%%%%%%%%%%%%%%%%%%%%%%%%%
%%%%%%%%%%%%%%%%%%%%%%%%%%%%%%%%%%%%%%%%%%%%%%%%%%%%%%%%%%%%%%%
%%%%%%%%%%%%%%%%%%%%%%%%%%%%%%%%%%%%%%%%%%%%%%%%%%%%%%%%%%%%%%%
%%%%%%%%%%%%%%%%%%%%%%%%%%%%%%%%%%%%%%%%%%%%%%%%%%%%%%%%%%%%%%%

We recall that $\langle \mathcal{M}_{k,\cdot}^{n} (\varphi) \rangle_{t} = t \mathcal{E}_{n} (\nabla^{n} \varphi^{n})$. From the Burkholder-Davis-Gundy inequality, it then follows that
\begin{align*}
    \E_{n} \Big[ \abs{\mathcal{M}_{k,t}^{n} (\varphi) - \mathcal{M}_{k,s}^{n} (\varphi)}^{p}  \Big] & \leq C_{p} \E_{n} \Big[ \langle \mathcal{M}_{k,\cdot}^{n} (\varphi) \rangle_{t-s}^{\frac{p}{2}}  \Big] 
    = C_{p} (t-s)^{\frac{p}{2}} \mathcal{E}_{n} (\nabla^{n} \varphi^{n})^{\frac{p}{2}}
\end{align*}
for all $p \geq 1$, some constant finite constant $C_p>0$. Tightness then follows from Kolmogorov's tightness criterion by taking $p > 3$.

%%%%%%%%%%%%%%%%%%%%%%%%%%%%%%%%%%%%%%%%%%%%%%%%%%%%%%%%%%%%%%%
%%%%%%%%%%%%%%%%%%%%%%%%%%%%%%%%%%%%%%%%%%%%%%%%%%%%%%%%%%%%%%%
%%%%%%%%%%%%%%%%%%%%%%%%%%%%%%%%%%%%%%%%%%%%%%%%%%%%%%%%%%%%%%%
%%%%%%%%%%%%%%%%%%%%%%%%%%%%%%%%%%%%%%%%%%%%%%%%%%%%%%%%%%%%%%%

\subsection{Symmetric term}

%%%%%%%%%%%%%%%%%%%%%%%%%%%%%%%%%%%%%%%%%%%%%%%%%%%%%%%%%%%%%%%
%%%%%%%%%%%%%%%%%%%%%%%%%%%%%%%%%%%%%%%%%%%%%%%%%%%%%%%%%%%%%%%
%%%%%%%%%%%%%%%%%%%%%%%%%%%%%%%%%%%%%%%%%%%%%%%%%%%%%%%%%%%%%%%
%%%%%%%%%%%%%%%%%%%%%%%%%%%%%%%%%%%%%%%%%%%%%%%%%%%%%%%%%%%%%%%
This term can be handled by an $L^2$ estimate:
\begin{align*}
     \E_{n} \Big[ \abs{\mathcal{S}_{k,t}^{n} (\varphi) - \mathcal{S}_{k,s}^{n} (\varphi)}^{2}  \Big] 
     & = \E_{n} \Big[ \Big| \int_{s}^{t}\frac{1}{2 n^{\frac{1}{2}}} \sum_{j \in \Z_{M}} u_{k,j}(\tau n^2) \Delta^{n} \varphi_{j}^{n}d \tau \Big|^{2}  \Big] 
     \leq
     \frac{\abs{t-s}^{2}}{4} \mathcal{E}_{n}(\Delta^{n} \varphi^{n}),
\end{align*}
where we used Jensen's inequality and the fact that $\{u_{k,j}\}_{j \in \Z_{M}}$ is an i.i.d. family of centered Gaussian random variables. Tightness then follows once again from Kolmogorov's criterion.

%%%%%%%%%%%%%%%%%%%%%%%%%%%%%%%%%%%%%%%%%%%%%%%%%%%%%%%%%%%%%%%
%%%%%%%%%%%%%%%%%%%%%%%%%%%%%%%%%%%%%%%%%%%%%%%%%%%%%%%%%%%%%%%
%%%%%%%%%%%%%%%%%%%%%%%%%%%%%%%%%%%%%%%%%%%%%%%%%%%%%%%%%%%%%%%
%%%%%%%%%%%%%%%%%%%%%%%%%%%%%%%%%%%%%%%%%%%%%%%%%%%%%%%%%%%%%%%

\subsection{Anti-symmetric term}

%%%%%%%%%%%%%%%%%%%%%%%%%%%%%%%%%%%%%%%%%%%%%%%%%%%%%%%%%%%%%%%
%%%%%%%%%%%%%%%%%%%%%%%%%%%%%%%%%%%%%%%%%%%%%%%%%%%%%%%%%%%%%%%
%%%%%%%%%%%%%%%%%%%%%%%%%%%%%%%%%%%%%%%%%%%%%%%%%%%%%%%%%%%%%%%
%%%%%%%%%%%%%%%%%%%%%%%%%%%%%%%%%%%%%%%%%%%%%%%%%%%%%%%%%%%%%%%
We state our main estimate on the anti-symmetric part.
%%%%%%%%%%%%%%%%%%%%%%%%%%%%%%%%%%%%%%%%%%%%%%%%%%%%%%%%%%%%%%%
%%%%%%%%%%%%%%%%%%%%%%%%%%%%%%%%%%%%%%%%%%%%%%%%%%%%%%%%%%%%%%%
\begin{proposition}
The anti-symmetric part of the dynamics satisfies
\begin{equation*}
    \E_{n} \Big[ \big| \mathcal{B}_{k,t}(\varphi) \big|^{2} \Big] \leq Ct^{\frac{3}{2}}
\end{equation*}
for every $\varphi \in \mathcal{S}(\T)$. Furthermore, each term in $\mathcal{B}_{k,t}(\varphi)$ satisfies the same bound by itself.
\end{proposition}
%%%%%%%%%%%%%%%%%%%%%%%%%%%%%%%%%%%%%%%%%%%%%%%%%%%%%%%%%%%%%%%
%%%%%%%%%%%%%%%%%%%%%%%%%%%%%%%%%%%%%%%%%%%%%%%%%%%%%%%%%%%%%%%
The rest of this section is devoted to the proof of this proposition. 
Recall that
\begin{align*}
    - \mathcal{B}_{k,t}^{n}(\varphi) 
    &= \int_{0}^{t}  \sum_{j \in \Z_{M}} G_{k,j}(sn^2) \nabla^{n} \varphi_{j}^{n} ds 
    \\
    & =  \alpha_{k} \underbrace{\int_{0}^{t} \sum_{j \in \Z_{M}} w_{k,j}(sn^2) \nabla^{n} \varphi_{j}^{n} ds}_{:=\textbf{W}_{k,t}^{n}(\varphi)} 
    + 
    \sum_{q \neq 0} \beta_{k}^{q} \underbrace{\int_{0}^{t} \sum_{j \in \Z_{M}} b_{k,j}^{q}(sn^2) \nabla^{n} \varphi_{j}^{n} ds}_{:=\textbf{B}_{k,t}^{n,q}(\varphi)}
    \\
    & \quad 
    + 
    \sum_{q \neq 0} \gamma_{k}^{q} \underbrace{ \int_{0}^{t} \sum_{j \in \Z_{M}} r_{k,j}^{q}(sn^2) \nabla^{n} \varphi_{j}^{n} ds}_{:=\textbf{R}_{k,t}^{n,q}(\varphi)} 
    + 
    \sum_{q \neq 0} \sum_{\substack{q' \neq 0 \\ q' \neq q}} \lambda_{k}^{k-q,k-q'} \underbrace{ \int_{0}^{t} \sum_{j \in \Z_{M}} p_{k,j}^{q,q'}(sn^2) \nabla^{n} \varphi_{j}^{n} ds}_{:=\textbf{P}_{k,t}^{n,q,q'}(\varphi)}.
\end{align*}
%%%%%%%%%%%%%%%%%%%%%%%%%%%%%%%%%%%%%%%%%%%%%%%%%%%%%%%%%%%%%%%
%%%%%%%%%%%%%%%%%%%%%%%%%%%%%%%%%%%%%%%%%%%%%%%%%%%%%%%%%%%%%%%
We begin with a lemma which will allow us to switch from terms involving $u_{k,j}^{2}$ to terms involving $u_{k,j} u_{k,j+1}$, to which we can apply the second-order Boltzmann-Gibbs principle. This will not be needed for products of terms depending on different components of the process.
%%%%%%%%%%%%%%%%%%%%%%%%%%%%%%%%%%%%%%%%%%%%%%%%%%%%%%%%%%%%%%%
%%%%%%%%%%%%%%%%%%%%%%%%%%%%%%%%%%%%%%%%%%%%%%%%%%%%%%%%%%%%%%%
\begin{lemma}
\label{lemma532}
Let
\begin{equation}
\label{YktWithOne}
    Y_{k,t}^{n} (\varphi) = \int_{0}^{t} \sum_{j \in \Z_{M}} \varphi_{j} \big\{u_{k,j}(sn^2) u_{k,j+1}(sn^2) - u_{k,j}^{2}(sn^2) +1 \big\} ds.
\end{equation}
There exists a finite constant $C>0$ such that
\begin{eqnarray}\label{eq:remainderY}
	\E_{n}\left[\left|
		Y_{k,t}^{n} (\varphi)
	\right|^2\right]
	\leq
	\frac{C t}{n} \mathcal{E}_{n}(\varphi).
\end{eqnarray}
In particular, $Y_{k,t}^{n} (\varphi)$ goes to zero in the ucp topology.
\end{lemma}
%%%%%%%%%%%%%%%%%%%%%%%%%%%%%%%%%%%%%%%%%%%%%%%%%%%%%%%%%%%%%%%
%%%%%%%%%%%%%%%%%%%%%%%%%%%%%%%%%%%%%%%%%%%%%%%%%%%%%%%%%%%%%%%
\begin{proof}
Using integration by parts,
\begin{align*}
    \int_{0}^{t} \sum_{j \in \Z_{M}} \varphi_{j} ( u_{k,j} u_{k,j+1} - u_{k,j}^{2}) f d\mu & = \int_{0}^{t} \sum_{j \in \Z_{M}} \varphi_{j} (u_{k,j+1} - u_{k,j}) u_{k,j} f d\mu \\
    & = \int_{0}^{t} \sum_{j \in \Z_{M}} \varphi_{j} (\partial_{k,j+1} - \partial_{k,j}) (u_{k,j} f) d\mu \\
    & = \int_{0}^{t} \sum_{j \in \Z_{M}} \varphi_{j} \{u_{k,j} ( \partial_{k,j+1} - \partial_{k,j})f - f\} d\mu.
\end{align*}
Hence,
\begin{equation*}
     \int_{0}^{t} \sum_{j \in \Z_{M}} \varphi_{j} \big\{u_{k,j} u_{k,j+1} - u_{k,j}^{2} +1 \big\} f d\mu = \int_{0}^{t} \sum_{j \in \Z_{M}} \varphi_{j} u_{k,j} ( \partial_{k,j+1} - \partial_{k,j})f  d\mu.
\end{equation*}
By Young's inequality,
\begin{align*}
    2 \int_{0}^{t} \sum_{j \in \Z_{M}} \varphi_{j} \big\{u_{k,j} u_{k,j+1} - u_{k,j}^{2} +1 \big\} f d\mu & = 2 \int_{0}^{t} \sum_{j \in \Z_{M}} \varphi_{j} u_{k,j} ( \partial_{k,j+1} - \partial_{k,j})f  d\mu \\
    & \leq \int_{0}^{t} \sum_{j \in \Z_{M}} \big\{ \alpha \varphi^{2} u_{k,j}^{2} + \frac{1}{\alpha}((\partial_{k,j+1} - \partial_{k,j})f)^{2} \big\} d\mu \\
    & = \int_{0}^{t} \sum_{j \in \Z_{M}} \big\{ \frac{2}{n^2} \varphi^{2} u_{k,j}^{2} + \frac{n^2}{2}((\partial_{k,j+1} - \partial_{k,j})f)^{2} \big\} d\mu \\
    & = \frac{2}{n} \mathcal{E}_{n}(\varphi) + \frac{n^2}{2} \sum_{j \in \Z_{M}} \int_{0}^{t}((\partial_{k,j+1} - \partial_{k,j})f)^{2} d\mu,
\end{align*}
by taking $\alpha = \frac{2}{n^2}$ and using that $u_{k,j} \sim \mathcal{N}(0,1)$. 
The estimate \eqref{eq:remainderY} then follows from the Kipnis-Varadhan estimate (Section \ref{kipnis-varadhan}).
We then obtain the convergence to $0$ in ucp topology by Chebyshev's inequality.
\end{proof}
%%%%%%%%%%%%%%%%%%%%%%%%%%%%%%%%%%%%%%%%%%%%%%%%%%%%%%%%%%%%%%%
%%%%%%%%%%%%%%%%%%%%%%%%%%%%%%%%%%%%%%%%%%%%%%%%%%%%%%%%%%%%%%%
This means that we can switch the term $w_{k,j}$ in the anti-symmetric part of the dynamics
by $u_{k,j} u_{k,j+1}$ modulo a vanishing term. Note that, as we apply the previous lemma to a gradient, the constant term $1$ in \eqref{YktWithOne} will disappear. Hence, to handle the term $\textbf{W}_{k,t}^{n}(\varphi)$ it is enough to prove the tightness of
\begin{equation*}
    \widetilde{\textbf{W}}_{k,t}^{n}(\varphi) := \int_{0}^{t} \sum_{j \in \Z_{M}} u_{k,j}(sn^2) u_{k,j+1}(sn^2) \nabla^{n} \varphi_{j}^{n} ds.
\end{equation*}
From the Boltzmann-Gibbs Principle (Proposition \ref{thm:BoltzmannGibbs}), we have that
\begin{equation}\label{eq:W-first-estimate}
    \E_{n} \Bigg[ \Bigg| \widetilde{\textbf{W}}_{k,t}^{n}(\varphi) - \int_{0}^{t} \sum_{j \in \Z_{M}} \overleftarrow{u}_{k,j}^{l}(sn^2) \overrightarrow{u}_{k,j}^{l}(sn^2) \nabla^{n} \varphi_{j}^{n} ds \Bigg|^{2} \Bigg] 
    \leq 
    C \frac{tl}{n} \mathcal{E}_{n}(\nabla^{n} \varphi^{n}).
\end{equation}
On the other hand, by Jensen's inequality,
\begin{align*}
    & \E_{n} \Bigg[ \Bigg| \int_{0}^{t} \sum_{j \in \Z_{M}} \overleftarrow{u}_{k,j}^{l}(sn^2) \overrightarrow{u}_{k,j}^{l}(sn^2) \nabla^{n} \varphi_{j}^{n} ds \Bigg|^{2} \Bigg] 
   \leq 
   t^{2} \E_{n} \Bigg[ \Bigg| \sum_{j \in \Z_{M}} \overleftarrow{u}_{k,j}^{l}(0) \overrightarrow{u}_{k,j}^{l}(0)  \nabla^{n} \varphi_{j}^{n} \Bigg|^{2}  \Bigg].
\end{align*}
Next, as in the proof of Lemma \ref{thm:one-block-forward},
a careful $L^{2}$ computation, taking dependencies into account, shows that
\begin{eqnarray*}
	\E_{n} \Bigg[ \Bigg| \sum_{j \in \Z_{M}} \overleftarrow{u}_{k,j}^{l}(0) \overrightarrow{u}_{k,j}^{l}(0)  \nabla^{n} \varphi_{j}^{n} \Bigg|^{2}  \Bigg]
	\leq
	\frac{n}{l}\mathcal{E}_{n}(\nabla^{n} \varphi^{n}).
\end{eqnarray*}
Hence,
\begin{align}\label{eq:W-second-estimate}
    & \E_{n} \Bigg[ \Bigg| \int_{0}^{t} \sum_{j \in \Z_{M}} \overleftarrow{u}_{k,j}^{l}(sn^2) \overrightarrow{u}_{k,j}^{l}(sn^2) \nabla^{n} \varphi_{j}^{n} ds \Bigg|^{2} \Bigg] 
   \leq 
   \frac{t^2 n}{l} \mathcal{E}_{n}(\nabla^{n} \varphi^{n}).
\end{align}
Combining \eqref{eq:W-first-estimate} and \eqref{eq:W-second-estimate}, we obtain
\begin{eqnarray*}
	 \E_{n} \Bigg[ \Bigg| \widetilde{\textbf{W}}_{k,t}^{n}(\varphi)  \Bigg|^{2} \Bigg] 
	 \leq
	 C \left(
	 	\frac{tl}{n}+\frac{t^2 n}{l}
	 \right),
\end{eqnarray*}
for some finite constant $C=C(\varphi)>0$. If $tn^2 \geq 1$, we choose $l \sim \sqrt{t}n$, which yields
\begin{eqnarray*}
	 \E_{n} \Bigg[ \Bigg| \widetilde{\textbf{W}}_{k,t}^{n}(\varphi)  \Bigg|^{2} \Bigg] 
	 \leq
	 C t^{3/2}.
\end{eqnarray*}
For $tn^2 \leq 1$, a crude $L^{2}$ bound gives
\begin{eqnarray*}
	\E_{n} \Bigg[ \Bigg| \widetilde{\textbf{W}}_{k,t}^{n}(\varphi)  \Bigg|^{2} \Bigg]
	&=&
	\E_{n} \Bigg[ \Bigg| \int_{0}^{t} \sum_{j \in \Z_{M}} u_{k,j}(sn) u_{k,j+1}(sn) \nabla^{n} \varphi_{j}^{n} ds \Bigg|^{2} \Bigg] 
	\\
	&\leq&
	C t^2 n \mathcal{E}_{n}(\nabla^{n} \varphi_{j}^{n})
	\leq
	C t^{3/2} \mathcal{E}_{n}(\nabla^{n} \varphi_{j}^{n}).
\end{eqnarray*}
This shows that $(\widetilde{\textbf{W}}_{k,t}^{n}(\varphi))_n$ is tight. 
%%%%%%%%%%%%%%%%%%%%%%%%%%%%%%%%%%%%%%%%%%%%%%%%%%%%%%%%%%%%%%%
%%%%%%%%%%%%%%%%%%%%%%%%%%%%%%%%%%%%%%%%%%%%%%%%%%%%%%%%%%%%%%%

We now consider the terms $\textbf{R}_{k,t}^{n,q}$.
We use the Boltzmann-Gibbs Principle (Proposition \ref{thm:BoltzmannGibbs}) once again to obtain
\begin{align*}
    & \E \Bigg[ \Bigg| \int_{0}^{t} \sum_{j \in \Z_{M}} u_{k-q,j}(sn^2) u_{k-q,j+1}(sn^2) \nabla^{n} \varphi_{j}^{n} ds - \int_{0}^{t} \sum_{j \in \Z_{M}} \overleftarrow{u}_{k-q,j}^{l}(sn^2) \overrightarrow{u}_{k,j}^{l}(sn^2) \nabla^{n} \varphi_{j}^{n} ds \Bigg|^{2} \Bigg] \\
    & \leq 
    C \frac{tl}{n} \mathcal{E}_{n}(\nabla^{n} \varphi^{n})
\end{align*}
and follow the same process as before. This yields the tightness of $(\textbf{R}_{k,t}^{n,q}(\varphi))_n$ for every $q \neq 0$. 

%%%%%%%%%%%%%%%%%%%%%%%%%%%%%%%%%%%%%%%%%%%%%%%%%%%%%%%%%%%%%%%
%%%%%%%%%%%%%%%%%%%%%%%%%%%%%%%%%%%%%%%%%%%%%%%%%%%%%%%%%%%%%%%
For the terms $\textbf{B}_{k,t}^{n,q}(\varphi)$ and $\textbf{P}_{k,t}^{n,q,q'}(\varphi)$, we use the Boltzmann-Gibbs Principle for crossed terms (Proposition \ref{thm:BoltzmannGibbsCrossed}) to obtain
\begin{equation*}
    \E \Bigg[ \Bigg| \int_{0}^{t} ds \sum_{j \in \Z_{M}} [u_{k,j}(sn^2) u_{\bar{k},j}(sn^2) -\overrightarrow{u}_{k,j-1}^{l}(sn^2) \overrightarrow{u}_{\bar{k},j-1}^{l}(sn^2) ] \nabla^{n} \varphi_{j}^{n} \Bigg|^{2} \Bigg] 
    \leq 
    C \frac{tl}{n} \mathcal{E}_{n}(\nabla^{n} \varphi^{n}),
\end{equation*}
for all $k\neq k'$.
Noticing that
\begin{equation*}
    \E \Bigg[ \Bigg| \int_{0}^{t} \sum_{j \in \Z_{M}} \overrightarrow{u}_{k,j-1}^{l}(sn^2) \overrightarrow{u}_{\bar{k},j-1}^{l}(sn^2) \nabla^{n} \varphi_{j}^{n} \Bigg|^{2} \Bigg] 
    \leq 
    \frac{t^{2}n}{l} \mathcal{E}_{n}(\nabla^{n} \varphi^{n}),
\end{equation*}
we conclude that
\begin{equation*}
    \E \Bigg[ \Bigg| \int_{0}^{t} \sum_{j \in \Z_{M}} u_{k,j}(sn^2) u_{\bar{k},j}(sn^2) \nabla^{n} \varphi_{j}^{n} ds \Bigg|^{2} \Bigg] 
    \leq 
    C \Bigg\{ \frac{tl}{n} + \frac{t^{2} n}{l} \Bigg\},
\end{equation*}
from which we get the tightness of $(\textbf{B}_{k,t}^{n,q}(\varphi))_n$ and $(\textbf{P}_{k,t}^{n,q,q'}(\varphi))_n$ following the same process we used for $\textbf{W}_{k,t}^{n}(\varphi)$. Finally, the tightness of each of its components yields tightness for $\mathcal{B}_{k,t}^{n}(\varphi)$.

%%%%%%%%%%%%%%%%%%%%%%%%%%%%%%%%%%%%%%%%%%%%%%%%%%%%%%%%%%%%%%%
%%%%%%%%%%%%%%%%%%%%%%%%%%%%%%%%%%%%%%%%%%%%%%%%%%%%%%%%%%%%%%%
%%%%%%%%%%%%%%%%%%%%%%%%%%%%%%%%%%%%%%%%%%%%%%%%%%%%%%%%%%%%%%%
%%%%%%%%%%%%%%%%%%%%%%%%%%%%%%%%%%%%%%%%%%%%%%%%%%%%%%%%%%%%%%%
%%%%%%%%%%%%%%%%%%%%%%%%%%%%%%%%%%%%%%%%%%%%%%%%%%%%%%%%%%%%%%%
%%%%%%%%%%%%%%%%%%%%%%%%%%%%%%%%%%%%%%%%%%%%%%%%%%%%%%%%%%%%%%%
%%%%%%%%%%%%%%%%%%%%%%%%%%%%%%%%%%%%%%%%%%%%%%%%%%%%%%%%%%%%%%%
%%%%%%%%%%%%%%%%%%%%%%%%%%%%%%%%%%%%%%%%%%%%%%%%%%%%%%%%%%%%%%%

\section{Identification of the limit}
\label{sec:identification}

%%%%%%%%%%%%%%%%%%%%%%%%%%%%%%%%%%%%%%%%%%%%%%%%%%%%%%%%%%%%%%%
%%%%%%%%%%%%%%%%%%%%%%%%%%%%%%%%%%%%%%%%%%%%%%%%%%%%%%%%%%%%%%%
%%%%%%%%%%%%%%%%%%%%%%%%%%%%%%%%%%%%%%%%%%%%%%%%%%%%%%%%%%%%%%%
%%%%%%%%%%%%%%%%%%%%%%%%%%%%%%%%%%%%%%%%%%%%%%%%%%%%%%%%%%%%%%%

From Section \ref{sec:tightness}, we get processes $\mathcal{S}$, $\textbf{W}$, $\textbf{R}$, $\textbf{B}$, $\textbf{P}$ and $\mathcal{M}$ such that
\begin{align*}
    \lim_{n \to \infty} \mathcal{M}_{k}^{n} & = \mathcal{M}_{k}, & \lim_{n \to \infty} \mathcal{S}_{k}^{n} & = \mathcal{S}_{k}, & \lim_{n \to \infty} \textbf{W}_{k}^{n} & = \textbf{W}_{k} \\
    \lim_{n \to \infty} \textbf{R}_{k}^{n,q} & = \textbf{R}_{k}^{q}, & \lim_{n \to \infty} \textbf{B}_{k}^{n,q} & = \textbf{B}_{k}^{q}, & \lim_{n \to \infty} \textbf{P}_{k}^{n,q,q'} & = \textbf{P}_{k}^{q,q'},
\end{align*}
along a subsequence that we still denote by $n$. We will now identify these limiting processes. By the convergence of these processes, we obtain that $\lim_{n \to \infty} \mathcal{X}_{k}^{n} = \mathcal{X}_{k}$, where $\mathcal{X}_{k}$ is a weighted sum of the previous processes.
Additionally, by Lemma \ref{lemma532}, $\lim_{n \to \infty} \widetilde{\textbf{W}}_{k}^{n} = \textbf{W}_{k}$.

%%%%%%%%%%%%%%%%%%%%%%%%%%%%%%%%%%%%%%%%%%%%%%%%%%%%%%%%%%%%%%%
%%%%%%%%%%%%%%%%%%%%%%%%%%%%%%%%%%%%%%%%%%%%%%%%%%%%%%%%%%%%%%%
%%%%%%%%%%%%%%%%%%%%%%%%%%%%%%%%%%%%%%%%%%%%%%%%%%%%%%%%%%%%%%%
%%%%%%%%%%%%%%%%%%%%%%%%%%%%%%%%%%%%%%%%%%%%%%%%%%%%%%%%%%%%%%%

\subsection{Convergence at fixed times}

%%%%%%%%%%%%%%%%%%%%%%%%%%%%%%%%%%%%%%%%%%%%%%%%%%%%%%%%%%%%%%%
%%%%%%%%%%%%%%%%%%%%%%%%%%%%%%%%%%%%%%%%%%%%%%%%%%%%%%%%%%%%%%%
%%%%%%%%%%%%%%%%%%%%%%%%%%%%%%%%%%%%%%%%%%%%%%%%%%%%%%%%%%%%%%%
%%%%%%%%%%%%%%%%%%%%%%%%%%%%%%%%%%%%%%%%%%%%%%%%%%%%%%%%%%%%%%%

We will show that $\mathcal{X}_{k,t}^{n}$ converges to a white noise for each fixed time $t \in [0, T]$ and each $k \in \Z_{K}$. Let $\varphi \in \mathcal{S}(\T)$. Recall that, for each fixed time $t \in [0, T]$, $(u_{k,j})_{k,j}$ is an i.i.d. family of standard Gaussian random variables. Hence,
\begin{align*}
    \lim_{n\to\infty}
    \E_{n} \Big[ \exp\left(
    		i \lambda \mathcal{X}_{k,t}^{n}(\varphi)
    	\right) \Big] 
    	& = 
    	\lim_{n\to\infty}
    	\E_{n} \left[ \exp\left(
    		i \lambda \frac{1}{n^{\frac{1}{2}}} \sum_{j \in \Z_{M}} u_{k,j}(tn) \varphi_{j}^{n}
    	\right) \right] 
    	\\
    & =
     \lim_{n\to\infty}
     \exp\left(
     	- \frac{\lambda^2}{2n}\sum_{j \in \Z_{M}} (\varphi_{j}^{n})^2
     \right) 
     \\
    & = 
     \lim_{n\to\infty}
     \exp\left(
     	- \frac{\lambda^{2}}{2} \mathcal{E}_{n}(\varphi)
     \right)
     =
     \exp\left(
     	- \frac{\lambda^{2}}{2} \int \varphi(x)^2 dx
     \right).
\end{align*}
Hence, $\mathcal{X}_{k,t}^{n}$ converges in distribution to a white noise. This in turns proves that $\mathcal{X}$ satisfies property (S).

%%%%%%%%%%%%%%%%%%%%%%%%%%%%%%%%%%%%%%%%%%%%%%%%%%%%%%%%%%%%%%%
%%%%%%%%%%%%%%%%%%%%%%%%%%%%%%%%%%%%%%%%%%%%%%%%%%%%%%%%%%%%%%%
%%%%%%%%%%%%%%%%%%%%%%%%%%%%%%%%%%%%%%%%%%%%%%%%%%%%%%%%%%%%%%%
%%%%%%%%%%%%%%%%%%%%%%%%%%%%%%%%%%%%%%%%%%%%%%%%%%%%%%%%%%%%%%%

\subsection{Martingale term} 

%%%%%%%%%%%%%%%%%%%%%%%%%%%%%%%%%%%%%%%%%%%%%%%%%%%%%%%%%%%%%%%
%%%%%%%%%%%%%%%%%%%%%%%%%%%%%%%%%%%%%%%%%%%%%%%%%%%%%%%%%%%%%%%
%%%%%%%%%%%%%%%%%%%%%%%%%%%%%%%%%%%%%%%%%%%%%%%%%%%%%%%%%%%%%%%
%%%%%%%%%%%%%%%%%%%%%%%%%%%%%%%%%%%%%%%%%%%%%%%%%%%%%%%%%%%%%%%

The quadratic variation of the martingale part satisfies
\begin{equation*}
    \lim_{n \to \infty} \langle \mathcal{M}_{k,\cdot}^{n}(\varphi) \rangle_{t} = \lim_{n \to \infty} t \mathcal{E}_{n} (\nabla^{n} \varphi^{n}) = t \int_{\T} (\partial_{x} \varphi(x))^{2}dx = t \left\|  \partial_{x} \varphi \right\|_{L^{2}}^{2},
\end{equation*}
for all $t \geq 0$ and $\varphi \in \mathcal{S}(\T)$. 
By a criterion of Aldous \cite{Aldous}, this implies that $(\mathcal{M}_{k,\cdot}^{n}(\varphi))_{n}$ converges to a Brownian motion with variance $\left\|  \partial_{x} \varphi  \right\|_{L^{2}}^{2}$. Hence, $(\mathcal{M}_{k,\cdot}^{n})_{n}$ converges in distribution to the derivative of a white noise.

%%%%%%%%%%%%%%%%%%%%%%%%%%%%%%%%%%%%%%%%%%%%%%%%%%%%%%%%%%%%%%%
%%%%%%%%%%%%%%%%%%%%%%%%%%%%%%%%%%%%%%%%%%%%%%%%%%%%%%%%%%%%%%%
%%%%%%%%%%%%%%%%%%%%%%%%%%%%%%%%%%%%%%%%%%%%%%%%%%%%%%%%%%%%%%%
%%%%%%%%%%%%%%%%%%%%%%%%%%%%%%%%%%%%%%%%%%%%%%%%%%%%%%%%%%%%%%%

\subsection{Symmetric term}

%%%%%%%%%%%%%%%%%%%%%%%%%%%%%%%%%%%%%%%%%%%%%%%%%%%%%%%%%%%%%%%
%%%%%%%%%%%%%%%%%%%%%%%%%%%%%%%%%%%%%%%%%%%%%%%%%%%%%%%%%%%%%%%
%%%%%%%%%%%%%%%%%%%%%%%%%%%%%%%%%%%%%%%%%%%%%%%%%%%%%%%%%%%%%%%
%%%%%%%%%%%%%%%%%%%%%%%%%%%%%%%%%%%%%%%%%%%%%%%%%%%%%%%%%%%%%%%

A second moment bound and a simple Taylor expansion show that
\begin{equation*}
   \E \Bigg[ \Bigg| \mathcal{S}_{k,t}^{n}(\varphi) - \frac{1}{2} \int_{0}^{t} \mathcal{X}_{k,s}^{n}(\partial_{x}^{2} \varphi)ds \Bigg|^{2} \Bigg] \leq \frac{Ct^{2}}{n^{2}},
\end{equation*}
for some finite constant $C=C(\varphi)>0$.
Together with the convergence of the fluctuation field, this shows that 
\begin{equation*}
    \mathcal{S}_{k}(\varphi) = \lim_{n \to \infty} \mathcal{S}_{k}^{n}(\varphi) = \frac{1}{2} \int_{0}^{\cdot}  \mathcal{X}_{k,s}(\partial_{x}^{2} \varphi)ds.
\end{equation*}

%%%%%%%%%%%%%%%%%%%%%%%%%%%%%%%%%%%%%%%%%%%%%%%%%%%%%%%%%%%%%%%
%%%%%%%%%%%%%%%%%%%%%%%%%%%%%%%%%%%%%%%%%%%%%%%%%%%%%%%%%%%%%%%
%%%%%%%%%%%%%%%%%%%%%%%%%%%%%%%%%%%%%%%%%%%%%%%%%%%%%%%%%%%%%%%
%%%%%%%%%%%%%%%%%%%%%%%%%%%%%%%%%%%%%%%%%%%%%%%%%%%%%%%%%%%%%%%

\subsection{Anti-symmetric term}

%%%%%%%%%%%%%%%%%%%%%%%%%%%%%%%%%%%%%%%%%%%%%%%%%%%%%%%%%%%%%%%
%%%%%%%%%%%%%%%%%%%%%%%%%%%%%%%%%%%%%%%%%%%%%%%%%%%%%%%%%%%%%%%
%%%%%%%%%%%%%%%%%%%%%%%%%%%%%%%%%%%%%%%%%%%%%%%%%%%%%%%%%%%%%%%
%%%%%%%%%%%%%%%%%%%%%%%%%%%%%%%%%%%%%%%%%%%%%%%%%%%%%%%%%%%%%%%
We will identify the limit of each term separately.
%%%%%%%%%%%%%%%%%%%%%%%%%%%%%%%%%%%%%%%%%%%%%%%%%%%%%%%%%%%%%%%
%%%%%%%%%%%%%%%%%%%%%%%%%%%%%%%%%%%%%%%%%%%%%%%%%%%%%%%%%%%%%%%
%%%%%%%%%%%%%%%%%%%%%%%%%%%%%%%%%%%%%%%%%%%%%%%%%%%%%%%%%%%%%%%
%%%%%%%%%%%%%%%%%%%%%%%%%%%%%%%%%%%%%%%%%%%%%%%%%%%%%%%%%%%%%%%

\subsubsection{Limit of $\textbf{W}_{k}^{n}$}

%%%%%%%%%%%%%%%%%%%%%%%%%%%%%%%%%%%%%%%%%%%%%%%%%%%%%%%%%%%%%%%
%%%%%%%%%%%%%%%%%%%%%%%%%%%%%%%%%%%%%%%%%%%%%%%%%%%%%%%%%%%%%%%
%%%%%%%%%%%%%%%%%%%%%%%%%%%%%%%%%%%%%%%%%%%%%%%%%%%%%%%%%%%%%%%
%%%%%%%%%%%%%%%%%%%%%%%%%%%%%%%%%%%%%%%%%%%%%%%%%%%%%%%%%%%%%%%

For the term $\textbf{W}_{k}^{n}$, it suffices to identify the limit of the process $\widetilde{\textbf{W}}_{k,\cdot}^{n}(\varphi)$. 
Remember that $\overleftarrow{\iota_{\varepsilon}}(x) = \mathds{1}_{(x-\varepsilon,x]}$ and $\overrightarrow{\iota_{\varepsilon}}(x) = \mathds{1}_{[x,x+\varepsilon)}$. Hence,
\begin{align*}
    \mathcal{X}_{k,t}^{n}(\overleftarrow{\iota_{\varepsilon}}(\tfrac{j}{n})) 
    = 
    n^{\frac{1}{2}} \overleftarrow{u}_{k,j}^{l}(tn^2),
    \quad
    \text{and}
    \quad
    \mathcal{X}_{k,t}^{n}(\overrightarrow{\iota_{\varepsilon}}(\tfrac{j}{n})) 
    = 
    n^{\frac{1}{2}} \overrightarrow{u}_{k,j}^{l}(tn^2).
\end{align*}
Then, if $l=tn$, we have
\begin{align*}
    \int_{0}^{t} \sum_{j \in \Z_{M}} \overleftarrow{u}_{k,j}^{l}(sn^2) \overrightarrow{u}_{k,j}^{l}(sn) \nabla \varphi_{j}^{n} ds 
    & = 
    \int_{0}^{t} \frac{1}{n} \sum_{j \in \Z_{n}} 
    \mathcal{X}_{k,s}^{n}(\overleftarrow{\iota_{\varepsilon}}(\tfrac{j}{n})) 
    \mathcal{X}_{k,s}^{n}(\overrightarrow{\iota_{\varepsilon}}(\tfrac{j}{n})) \nabla \varphi_{j}^{n} ds 
    \\
    & \myrightarrow{n \to \infty} 
    \int_{0}^{t} \int_{\T} \mathcal{X}_{k,s}(\overleftarrow{\iota_{\varepsilon}}(x)) \mathcal{X}_{k,s}(\overrightarrow{\iota_{\varepsilon}}(x)) \partial_{x} \varphi(x) dx ds 
    \\
    & := \mathcal{A}_{0,t}^{\varepsilon,(k,k)}(\varphi).
\end{align*}
%%%%%%%%%%%%%%%%%%%%%%%%%%%%%%%%%%%%%%%%%%%%%%%%%%%%%%%%%%%%%%%
%%%%%%%%%%%%%%%%%%%%%%%%%%%%%%%%%%%%%%%%%%%%%%%%%%%%%%%%%%%%%%%
\begin{remark}
Given that $\overleftarrow{\iota_{\varepsilon}}(x),\overrightarrow{\iota_{\varepsilon}}(x) \not \in \mathcal{S}(\R)$, the limit does not follow immediately from the convergence of the field. However, it 
follows by a suitable approximation by $\mathcal{S}(\R)$ functions (see \cite{goncalvesJaraEnergysol}, Section 5.3 for details).
\end{remark}
%%%%%%%%%%%%%%%%%%%%%%%%%%%%%%%%%%%%%%%%%%%%%%%%%%%%%%%%%%%%%%%
%%%%%%%%%%%%%%%%%%%%%%%%%%%%%%%%%%%%%%%%%%%%%%%%%%%%%%%%%%%%%%%
Now, we show that $\mathcal{A}^{\varepsilon,(k,k)}(\varphi)$ satisfies the energy condition \eqref{conditionEC}. From the Boltzmann-Gibbs Principle (Proposition \ref{thm:BoltzmannGibbs}) and stationarity, we have that
\begin{equation*}
    \E \Bigg[ \Bigg| \widetilde{\textbf{W}}_{k,t}^{n}(\varphi)-\widetilde{\textbf{W}}_{k,s}^{n}(\varphi)  
    - 
    \int_{s}^{t} \sum_{j \in \Z_{n}} \overleftarrow{u}_{k,j}^{l}(sn^2) \overrightarrow{u}_{k,j}^{l}(sn^2) \nabla^{n} \varphi_{j}^{n} ds \Bigg|^{2} \Bigg] 
    \leq 
    C \frac{(t-s)l}{n} \mathcal{E}_{n}(\nabla^{n} \varphi^{n}).
\end{equation*}
Taking $l = \varepsilon n$, we obtain
\begin{equation*}
    \E \Bigg[ \Bigg| \widetilde{\textbf{W}}_{k,t}^{n}(\varphi)-\widetilde{\textbf{W}}_{k,s}^{n}(\varphi)  - \int_{s}^{t} \sum_{j \in \Z_{M}} \overleftarrow{u}_{k,j}^{l}(sn) \overrightarrow{u}_{k,j}^{l}(sn) \nabla^{n} \varphi_{j}^{n} ds \Bigg|^{2} \Bigg] \leq C (t-s) \varepsilon,
\end{equation*}
from some finite constant $C=C(\varphi)>0$.
Therefore, by Fatou's lemma,
\begin{align}\label{eq:W-intermediate}
    \E \Bigg[ \Bigg| \textbf{W}_{k,t}(\varphi)-\textbf{W}_{k,s}(\varphi)  - \mathcal{A}_{s,t}^{\varepsilon,(k,k)}(\varphi) \Bigg|^{2} \Bigg] 
    \leq C (t-s) \varepsilon.
\end{align}
In the same way, if $0 < \delta < \varepsilon$,
\begin{equation*}
\E \Bigg[ \Bigg| \textbf{W}_{k,t}(\varphi)-\textbf{W}_{k,s}(\varphi)  - \mathcal{A}_{s,t}^{\delta,(k,k)}(\varphi) \Bigg|^{2} \Bigg] \leq C(t-s) \delta.
\end{equation*}
Collecting the last two estimates, we obtain
\begin{align*}
    \E \Bigg[ \Bigg| \mathcal{A}_{s,t}^{\varepsilon,(k,k)}(\varphi)  - \mathcal{A}_{s,t}^{\delta,(k,k)}(\varphi) \Bigg|^{2} \Bigg] 
    \leq C(t-s) \varepsilon,
\end{align*}
for all $0 < \delta < \varepsilon$.
This shows that $\mathcal{A}^{\varepsilon,(k,k)}$ satisfies the energy estimate \eqref{conditionEC} for all $k \in \Z_{K}$. Hence, by Theorem \ref{ThmAconvergence}, taking $\varepsilon \to 0$, $\mathcal{A}_{s,t}^{\varepsilon,(k,k)}(\varphi)$ converges in $L^{2}$ to a process we denote by $\mathcal{A}_{s,t}^{(k,k)}(\varphi)$. Finally, by \eqref{eq:W-intermediate} and Fatou's lemma,
\begin{align*}
     \E \Bigg[ \Bigg| \textbf{W}_{k,t}(\varphi)-\textbf{W}_{k,s}(\varphi) - \mathcal{A}_{s,t}^{(k,k)}(\varphi) \Bigg|^{2} \Bigg] 
     & \leq 
     \lim_{\varepsilon \to 0} \E \Bigg[ \Bigg| \textbf{W}_{k,t}(\varphi)-\textbf{W}_{k,s}(\varphi)  - \mathcal{A}_{s,t}^{\varepsilon,(k,k)}(\varphi) \Bigg|^{2} \Bigg] 
     = 0,
\end{align*}
concluding that the limit of the process $\textbf{W}_{k,\cdot}^{n}(\varphi)$ is $\mathcal{A}_{\cdot}^{(k,k)}(\varphi)$.

%%%%%%%%%%%%%%%%%%%%%%%%%%%%%%%%%%%%%%%%%%%%%%%%%%%%%%%%%%%%%%%
%%%%%%%%%%%%%%%%%%%%%%%%%%%%%%%%%%%%%%%%%%%%%%%%%%%%%%%%%%%%%%%
%%%%%%%%%%%%%%%%%%%%%%%%%%%%%%%%%%%%%%%%%%%%%%%%%%%%%%%%%%%%%%%
%%%%%%%%%%%%%%%%%%%%%%%%%%%%%%%%%%%%%%%%%%%%%%%%%%%%%%%%%%%%%%%

\subsubsection{Limit of $\textbf{B}_{k}^{n,q}$}

%%%%%%%%%%%%%%%%%%%%%%%%%%%%%%%%%%%%%%%%%%%%%%%%%%%%%%%%%%%%%%%
%%%%%%%%%%%%%%%%%%%%%%%%%%%%%%%%%%%%%%%%%%%%%%%%%%%%%%%%%%%%%%%
%%%%%%%%%%%%%%%%%%%%%%%%%%%%%%%%%%%%%%%%%%%%%%%%%%%%%%%%%%%%%%%
%%%%%%%%%%%%%%%%%%%%%%%%%%%%%%%%%%%%%%%%%%%%%%%%%%%%%%%%%%%%%%%

Recall that
\begin{equation*}
    \textbf{B}_{k,t}^{n,q}(\varphi) = \int_{0}^{t} \sum_{j \in \Z_{M}} b_{k,j}^{q}(sn^2) \nabla^{n} \varphi_{j}^{n} ds
\end{equation*}
where
\begin{equation*}
    b_{k,j}^{q} = \frac{1}{2} (u_{k,j}u_{k+q,j} + u_{k,j+1}u_{k+q,j+1}).
\end{equation*}
By a simple $L^2$ computation, we note that it is enough to identify the limit of the processes
\begin{equation*}
    U_{k,\bar{k}}^{t,n}(\varphi) = \int_{0}^{t} \sum_{j \in \Z_{M}} u_{k,j}(sn^2)u_{\bar{k},j}(sn^2) \nabla^{n} \varphi_{j}^{n} ds, 
\end{equation*}
for $k \neq \bar{k}$. We denote $U_{k,\bar{k}}^{t}(\varphi) = \lim_{n\to\infty} U_{k,\bar{k}}^{t,n}(\varphi)$. As above, it holds that
\begin{align*}
    \int_{0}^{t} \sum_{j \in \Z_{M}} \overrightarrow{u}_{k,j-1}^{l}(sn^2) \overrightarrow{u}_{\bar{k},j-1}^{l}(sn^2) \nabla^{n} \varphi_{j}^{n} ds 
    & = 
    \int_{0}^{t} \frac{1}{\sqrt{n}} \sum_{j \in \Z_{n}} 
    \mathcal{X}_{k,s}^{n}(\overrightarrow{\iota_{\varepsilon}}(\tfrac{j-1}{\sqrt{n}})) \mathcal{X}_{\bar{k},s}^{n}(\overrightarrow{\iota_{\varepsilon}}(\tfrac{j-1}{\sqrt{n}}))  \nabla^{n} \varphi_{j}^{n} ds 
    \\
    & \myrightarrow{n \to \infty} 
    \int_{0}^{t} \int_{\T}  \mathcal{X}_{k,s}(\overrightarrow{\iota_{\varepsilon}}(x)) \mathcal{X}_{\bar{k},s}(\overrightarrow{\iota_{\varepsilon}}(x))  \partial_{x} \varphi(x) dx ds
    \\
    & =: \mathcal{A}_{0,t}^{\varepsilon,(k,\bar{k})}(\varphi).
\end{align*}
From the Boltzman-Gibbs Principle for crossed terms (Proposition \ref{thm:BoltzmannGibbsCrossed}), we have that
\begin{equation*}
    \E \Bigg[ \Bigg| U_{k,\bar{k}}^{t,n}(\varphi) 
    - \int_{0}^{t} ds \sum_{j \in \Z_{n}} \overrightarrow{u}_{k,j-1}^{l}(sn^2) \overrightarrow{u}_{\bar{k},j-1}^{l}(sn^2) \nabla^{n} \varphi_{j}^{n} 
     \Bigg|^{2} \Bigg] 
     \leq 
     C \frac{tl}{n} \mathcal{E}_{n}(\nabla^{n} \varphi^{n}).
\end{equation*}
Hence, taking $l = \varepsilon n$, applying Fatou's lemma and using stationarity, we obtain that
\begin{align*}
    \E \Bigg[ \Bigg| U_{k,\bar{k}}^{t}(\varphi)-U_{k,k+q}^{s}(\varphi)  - \mathcal{A}_{s,t}^{\varepsilon,(k,\bar{k})}(\varphi) \Bigg|^{2} \Bigg] 
    \leq C (t-s) \varepsilon.
\end{align*}
As above, this shows that $\mathcal{A}_{s,t}^{\varepsilon,(k,\bar{k})}(\varphi)$ satisfies the energy condition. Hence, $\mathcal{A}_{s,t}^{(k,\bar{k})}(\varphi) = \lim_{\varepsilon\to0} \mathcal{A}_{s,t}^{\varepsilon,(k,\bar{k})}(\varphi)$ exists and, proceeding again as above, we conclude that $\lim_{n\to0} U^{t,n}_{k,\bar{k}} = \mathcal{A}_{s,t}^{(k,\bar{k})}(\varphi)$.

In particular, we conclude that the limit of the process $\textbf{B}_{k,\cdot}^{n,q}(\varphi)$ is $\mathcal{A}_{\cdot}^{(k,k+q)}(\varphi)$.

%%%%%%%%%%%%%%%%%%%%%%%%%%%%%%%%%%%%%%%%%%%%%%%%%%%%%%%%%%%%%%%
%%%%%%%%%%%%%%%%%%%%%%%%%%%%%%%%%%%%%%%%%%%%%%%%%%%%%%%%%%%%%%%
%%%%%%%%%%%%%%%%%%%%%%%%%%%%%%%%%%%%%%%%%%%%%%%%%%%%%%%%%%%%%%%
%%%%%%%%%%%%%%%%%%%%%%%%%%%%%%%%%%%%%%%%%%%%%%%%%%%%%%%%%%%%%%%

\subsubsection{Limit of $\textbf{R}_{k}^{n,q}$}

%%%%%%%%%%%%%%%%%%%%%%%%%%%%%%%%%%%%%%%%%%%%%%%%%%%%%%%%%%%%%%%
%%%%%%%%%%%%%%%%%%%%%%%%%%%%%%%%%%%%%%%%%%%%%%%%%%%%%%%%%%%%%%%
%%%%%%%%%%%%%%%%%%%%%%%%%%%%%%%%%%%%%%%%%%%%%%%%%%%%%%%%%%%%%%%
%%%%%%%%%%%%%%%%%%%%%%%%%%%%%%%%%%%%%%%%%%%%%%%%%%%%%%%%%%%%%%%

Recall that
\begin{equation*}
    \textbf{R}_{k,t}^{n,q}(\varphi) = \int_{0}^{t} \sum_{j \in \Z_{n}} r_{k,j}^{q}(sn^2) \nabla^{n} \varphi_{j}^{n} ds
\end{equation*}
where
\begin{equation*}
    r_{k,j}^{q} = u_{k-q,j}u_{k-q,j+1}.
\end{equation*}
Following the argument used to compute the limit of the term $\widetilde{\textbf{W}}_{k,\cdot}^{n}(\varphi)$, we obtain that the limit of the process $\textbf{R}_{k,\cdot}^{n,q}(\varphi)$ is $\mathcal{A}_{\cdot}^{(k-q,k-q)}(\varphi)$.

%%%%%%%%%%%%%%%%%%%%%%%%%%%%%%%%%%%%%%%%%%%%%%%%%%%%%%%%%%%%%%%
%%%%%%%%%%%%%%%%%%%%%%%%%%%%%%%%%%%%%%%%%%%%%%%%%%%%%%%%%%%%%%%
%%%%%%%%%%%%%%%%%%%%%%%%%%%%%%%%%%%%%%%%%%%%%%%%%%%%%%%%%%%%%%%
%%%%%%%%%%%%%%%%%%%%%%%%%%%%%%%%%%%%%%%%%%%%%%%%%%%%%%%%%%%%%%%

\subsubsection{Limit of $\textbf{P}_{k}^{n,q,q'}$}

%%%%%%%%%%%%%%%%%%%%%%%%%%%%%%%%%%%%%%%%%%%%%%%%%%%%%%%%%%%%%%%
%%%%%%%%%%%%%%%%%%%%%%%%%%%%%%%%%%%%%%%%%%%%%%%%%%%%%%%%%%%%%%%
%%%%%%%%%%%%%%%%%%%%%%%%%%%%%%%%%%%%%%%%%%%%%%%%%%%%%%%%%%%%%%%
%%%%%%%%%%%%%%%%%%%%%%%%%%%%%%%%%%%%%%%%%%%%%%%%%%%%%%%%%%%%%%%

Recall that
\begin{equation*}
    \textbf{P}_{k,t}^{n,q,q'}(\varphi) = \int_{0}^{t} \sum_{j \in \Z_{M}} p_{k,j}^{q,q'}(sn^2) \nabla^{n} \varphi_{j}^{n} ds
\end{equation*}
where
\begin{equation*}
    p_{k,j}^{q,q'} = \frac{1}{6} ( 2u_{k-q,j}u_{k-q',j} +u_{k-q,j}u_{k-q',j+1} + u_{k-q,j+1}u_{k-q',j} + 2u_{k-q,j+1}u_{k-q',j+1}).
\end{equation*}
Following the argument used to compute the limit of the term $\textbf{B}_{k,\cdot}^{n,q}(\varphi)$, we obtain that the limit of the process $\textbf{P}_{k,\cdot}^{n,q,q'}(\varphi)$ is $\mathcal{A}_{\cdot}^{(k-q,k-q')}(\varphi)$.

%%%%%%%%%%%%%%%%%%%%%%%%%%%%%%%%%%%%%%%%%%%%%%%%%%%%%%%%%%%%%%%
%%%%%%%%%%%%%%%%%%%%%%%%%%%%%%%%%%%%%%%%%%%%%%%%%%%%%%%%%%%%%%%
\begin{remark}
	As all our arguments apply to the reversed process with straightforward modification, this finishes the proof of Theorem \ref{mainResult}.
\end{remark}
%%%%%%%%%%%%%%%%%%%%%%%%%%%%%%%%%%%%%%%%%%%%%%%%%%%%%%%%%%%%%%%
%%%%%%%%%%%%%%%%%%%%%%%%%%%%%%%%%%%%%%%%%%%%%%%%%%%%%%%%%%%%%%%

%%%%%%%%%%%%%%%%%%%%%%%%%%%%%%%%%%%%%%%%%%%%%%%%%%%%%%%%%%%%%%%
%%%%%%%%%%%%%%%%%%%%%%%%%%%%%%%%%%%%%%%%%%%%%%%%%%%%%%%%%%%%%%%
%%%%%%%%%%%%%%%%%%%%%%%%%%%%%%%%%%%%%%%%%%%%%%%%%%%%%%%%%%%%%%%
%%%%%%%%%%%%%%%%%%%%%%%%%%%%%%%%%%%%%%%%%%%%%%%%%%%%%%%%%%%%%%%
%%%%%%%%%%%%%%%%%%%%%%%%%%%%%%%%%%%%%%%%%%%%%%%%%%%%%%%%%%%%%%%
%%%%%%%%%%%%%%%%%%%%%%%%%%%%%%%%%%%%%%%%%%%%%%%%%%%%%%%%%%%%%%%
%%%%%%%%%%%%%%%%%%%%%%%%%%%%%%%%%%%%%%%%%%%%%%%%%%%%%%%%%%%%%%%
%%%%%%%%%%%%%%%%%%%%%%%%%%%%%%%%%%%%%%%%%%%%%%%%%%%%%%%%%%%%%%%

\appendix

%%%%%%%%%%%%%%%%%%%%%%%%%%%%%%%%%%%%%%%%%%%%%%%%%%%%%%%%%%%%%%%
%%%%%%%%%%%%%%%%%%%%%%%%%%%%%%%%%%%%%%%%%%%%%%%%%%%%%%%%%%%%%%%
%%%%%%%%%%%%%%%%%%%%%%%%%%%%%%%%%%%%%%%%%%%%%%%%%%%%%%%%%%%%%%%
%%%%%%%%%%%%%%%%%%%%%%%%%%%%%%%%%%%%%%%%%%%%%%%%%%%%%%%%%%%%%%%

\section{Proof of Lemma \ref{thm:equivalence-trilinear} }
\label{sec:proof-trilinear}

%%%%%%%%%%%%%%%%%%%%%%%%%%%%%%%%%%%%%%%%%%%%%%%%%%%%%%%%%%%%%%%
%%%%%%%%%%%%%%%%%%%%%%%%%%%%%%%%%%%%%%%%%%%%%%%%%%%%%%%%%%%%%%%
%%%%%%%%%%%%%%%%%%%%%%%%%%%%%%%%%%%%%%%%%%%%%%%%%%%%%%%%%%%%%%%
%%%%%%%%%%%%%%%%%%%%%%%%%%%%%%%%%%%%%%%%%%%%%%%%%%%%%%%%%%%%%%%

%%%%%%%%%%%%%%%%%%%%%%%%%%%%%%%%%%%%%%%%%%%%%%%%%%%%%%%%%%%%%%%
%%%%%%%%%%%%%%%%%%%%%%%%%%%%%%%%%%%%%%%%%%%%%%%%%%%%%%%%%%%%%%%
\begin{proof}[Proof of Lemma \ref{thm:equivalence-trilinear}]
First, we show that conditions \eqref{betaGamma}-\eqref{lambda} imply the trilinear condition. 
The equations for $u_{k}$, $u_{k+a}$, $u_{k-a}$ are
\begin{eqnarray*}
   \partial_{t} u_{k} 
   &=& 
   \frac{1}{2}\partial_{x}^{2} u_{k} + \partial_{x} \big(\alpha_{k} u_{k}^{2} + \sum_{l \neq 0} \big\{ \beta_{k}^{l} u_{k} u_{k+l} + \gamma_{k}^{l} u_{k-l}^{2} + \sum_{\substack{l' \neq 0 \\ l' \neq l}} \lambda_{k}^{k-l,k-l'} u_{k-l}u_{k-l'} \big\} \big) + \partial_{x} \mathscr{W}_{k},
 \\
    \partial_{t} u_{k+a} 
    &=&  \frac{1}{2}\partial_{x}^{2} u_{k+a} + \partial_{x} \big(\alpha_{k+a} u_{k+a}^{2} + \sum_{l \neq 0} \big\{ \beta_{k+a}^{l} u_{k+a} u_{k+a+l} + \gamma_{k+a}^{l} u_{k+a-l}^{2} 
    \\
   && \phantom{blablablablabla}
   + \sum_{\substack{l' \neq 0 \\ l' \neq l}} \lambda_{k+a}^{k+a-l,k+a-l'} u_{k+a-l}u_{k+a-l'} \big\} \big) + \partial_{x} \mathscr{W}_{k+a},
   \\
   \\
   \\
	 \end{eqnarray*}
  \begin{eqnarray*}
    \partial_{t} u_{k-a} 
    &=& 
    \frac{1}{2}\partial_{x}^{2} u_{k-a} + \partial_{x} \big(\alpha_{k-a} u_{k-a}^{2} + \sum_{l \neq 0} \big\{ \beta_{k-a}^{l} u_{k-a} u_{k-a+l} + \gamma_{k-a}^{l} u_{k-a-l}^{2} \\
   && \phantom{blablablablabla}
   + \sum_{\substack{l' \neq 0 \\ l' \neq l}} \lambda_{k-a}^{k-a-l,k-a-l'} u_{k-a-l}u_{k-a-l'} \big\} \big) + \partial_{x} \mathscr{W}_{k-a}.
\end{eqnarray*}
%\begin{equation*}
%   \partial_{t} u_{k} = \frac{1}{2}\partial_{x}^{2} u_{k} + \partial_{x} \big(\alpha_{k} u_{k}^{2} + \sum_{l \neq 0} \big\{ \beta_{k}^{l} u_{k} u_{k+l} + \gamma_{k}^{l} u_{k-l}^{2} + \sum_{\substack{l' \neq 0 \\ l' \neq l}} \lambda_{k}^{k-l,k-l'} u_{k-l}u_{k-l'} \big\} \big) + \partial_{x} \mathscr{W}_{k},
%\end{equation*}
%\begin{align*}
%    \partial_{t} u_{k+a} &=  \frac{1}{2}\partial_{x}^{2} u_{k+a} + \partial_{x} \big(\alpha_{k+a} u_{k+a}^{2} + \sum_{l \neq 0} \big\{ \beta_{k+a}^{l} u_{k+a} u_{k+a+l} + \gamma_{k+a}^{l} u_{k+a-l}^{2} \\
%   & \quad + \sum_{\substack{l' \neq 0 \\ l' \neq l}} \lambda_{k+a}^{k+a-l,k+a-l'} u_{k+a-l}u_{k+a-l'} \big\} \big) + \partial_{x} \mathscr{W}_{k+a},
%\end{align*}
%\begin{align*}
%    \partial_{t} u_{k-a} &= \frac{1}{2}\partial_{x}^{2} u_{k-a} + \partial_{x} \big(\alpha_{k-a} u_{k-a}^{2} + \sum_{l \neq 0} \big\{ \beta_{k-a}^{l} u_{k-a} u_{k-a+l} + \gamma_{k-a}^{l} u_{k-a-l}^{2} \\
%   & \quad + \sum_{\substack{l' \neq 0 \\ l' \neq l}} \lambda_{k-a}^{k-a-l,k-a-l'} u_{k-a-l}u_{k-a-l'} \big\} \big) + \partial_{x} \mathscr{W}_{k-a}.
%\end{align*}
Hence,
\begin{equation*}
    \Gamma_{k,k}^{k} = \alpha_{k}, \ \Gamma_{k+a,k+a}^{k+a} = \alpha_{k+a}, \ \Gamma_{k-a,k-a}^{k-a} = \alpha_{k-a},
\end{equation*}
\begin{equation*}
    \Gamma_{k,k+a}^{k} = \Gamma_{k+a,k}^{k} = \frac{\beta_{k}^{a}}{2} = \Gamma_{k,k}^{k+a} = \gamma_{k+a}^{a},
\end{equation*}
\begin{equation*}
    \Gamma_{k-a,k-a}^{k} = \gamma_{k}^{a} = \Gamma_{k-a,k}^{k-a} = \Gamma_{k,k-a}^{k-a} = \frac{\beta_{k-a}^{a}}{2},
\end{equation*}
\begin{equation*}
    \Gamma_{k-a,k-a'}^{k} = \lambda_{k}^{k-a,k-a'} = \Gamma_{k-a',k-a}^{k} = \lambda_{k}^{k-a',k-a} = \Gamma_{k,k-a'}^{k-a} = \lambda_{k-a}^{k,k-a'}.
\end{equation*}
The trilinear condition simply follows from the identity $\gamma_{k+a}^{a} = \frac{\beta_{k}^{a}}{2}$ for every $a,k \in \Z_{K}$.

Now, consider a system of coupled Burgers equations satisfying the trilinear condition,
\begin{equation*}
    \partial_{t} u_{k} = \frac{1}{2} \partial_{x}^{2} u_{k} + \sum_{i,j \in \Z_{K}} \Gamma_{i,j}^{k} \partial_{x} (u_{i} u_{j}) + \partial_{x} \mathscr{W}_{k}, \quad k \in \Z_{K},
\end{equation*}
\begin{equation*}
    \Gamma_{i,j}^{k} = \Gamma_{j,i}^{k} = \Gamma_{k,j}^{i} \quad \text{for all } i,j,k \in \Z_{K}.
\end{equation*}
We rewrite the equations for $u_{k}$, $u_{k+a}$, $u_{k-a}$ as
\begin{eqnarray*}
   \partial_{t} u_{k} 
   &=& 
   \frac{1}{2}\partial_{x}^{2} u_{k} + \partial_{x} \big(\Gamma_{k,k}^{k} u_{k}^{2} + \sum_{l \neq 0} \big\{ \Gamma_{k,k+l}^{k} u_{k} u_{k+l} + \Gamma_{k+l,k}^{k} u_{k} u_{k+l} 
   \\
   && \quad 
   + \Gamma_{k-l,k-l}^{k} u_{k-l}^{2} + \sum_{\substack{l' \neq 0 \\ l' \neq l}} \Gamma_{k-l,k-l'}^{k} u_{k-l}u_{k-l'} \big\} \big) + \partial_{x} \mathscr{W}_{k},
	\\
   \partial_{t} u_{k+a} 
   &=& 
   \frac{1}{2}\partial_{x}^{2} u_{k+a} + \partial_{x} \Big(\Gamma_{k+a,k+a}^{k+a} u_{k+a}^{2} + \sum_{l \neq 0} \Big\{ \Gamma_{k+a,k+a+l}^{k+a} u_{k+a} u_{k+a+l} 
   \\
   && \quad 
   + \Gamma_{k+a+l,k+a}^{k+a} u_{k+a} u_{k+a+l} + \Gamma_{k+a-l,k+a-l}^{k} u_{k+a-l}^{2}  
   \\
   && \quad + \sum_{\substack{l' \neq 0 \\ l' \neq l}} \Gamma_{k+a-l,k+a-l'}^{k+a} u_{k+a-l}u_{k+a-l'} \Big\} \Big) + \partial_{x} \mathscr{W}_{k+a},
	\\
   \partial_{t} u_{k-a} 
   &=& 
   \frac{1}{2}\partial_{x}^{2} u_{k-a} + \partial_{x} \Big(\Gamma_{k-a,k-a}^{k-a} u_{k-a}^{2} + \sum_{l \neq 0} \Big\{ \Gamma_{k-a,k-a+l}^{k-a} u_{k-a} u_{k-a+l} 
   \\
   && \quad 
   + \Gamma_{k-a+l,k-a}^{k-a} u_{k-a} u_{k-a+l} + \Gamma_{k-a-l,k-a-l}^{k-a} u_{k-a-l}^{2} 
   \\
   && \quad 
   + \sum_{\substack{l' \neq 0 \\ l' \neq l}} \Gamma_{k-a-l,k-a-l'}^{k-a} u_{k-a-l}u_{k-a-l'} \Big\} \Big) + \partial_{x} \mathscr{W}_{k-a},
\end{eqnarray*}
%\begin{align*}
%   \partial_{t} u_{k} & = \frac{1}{2}\partial_{x}^{2} u_{k} + \partial_{x} \big(\Gamma_{k,k}^{k} u_{k}^{2} + \sum_{l \neq 0} \big\{ \Gamma_{k,k+l}^{k} u_{k} u_{k+l} + \Gamma_{k+l,k}^{k} u_{k} u_{k+l} \\
%   & \quad + \Gamma_{k-l,k-l}^{k} u_{k-l}^{2} + \sum_{\substack{l' \neq 0 \\ l' \neq l}} \Gamma_{k-l,k-l'}^{k} u_{k-l}u_{k-l'} \big\} \big) + \partial_{x} \mathscr{W}_{k},
%\end{align*}
%\begin{align*}
%   \partial_{t} u_{k+a} & = \frac{1}{2}\partial_{x}^{2} u_{k+a} + \partial_{x} \Big(\Gamma_{k+a,k+a}^{k+a} u_{k+a}^{2} + \sum_{l \neq 0} \Big\{ \Gamma_{k+a,k+a+l}^{k+a} u_{k+a} u_{k+a+l} \\
%   & \quad + \Gamma_{k+a+l,k+a}^{k+a} u_{k+a} u_{k+a+l} + \Gamma_{k+a-l,k+a-l}^{k} u_{k+a-l}^{2}  \\
%   & \quad + \sum_{\substack{l' \neq 0 \\ l' \neq l}} \Gamma_{k+a-l,k+a-l'}^{k+a} u_{k+a-l}u_{k+a-l'} \Big\} \Big) + \partial_{x} \mathscr{W}_{k+a},
%\end{align*}
%\begin{align*}
%   \partial_{t} u_{k-a} & = \frac{1}{2}\partial_{x}^{2} u_{k-a} + \partial_{x} \Big(\Gamma_{k-a,k-a}^{k-a} u_{k-a}^{2} + \sum_{l \neq 0} \Big\{ \Gamma_{k-a,k-a+l}^{k-a} u_{k-a} u_{k-a+l} \\
%   & \quad + \Gamma_{k-a+l,k-a}^{k-a} u_{k-a} u_{k-a+l} + \Gamma_{k-a-l,k-a-l}^{k-a} u_{k-a-l}^{2} \\
%   & \quad + \sum_{\substack{l' \neq 0 \\ l' \neq l}} \Gamma_{k-a-l,k-a-l'}^{k-a} u_{k-a-l}u_{k-a-l'} \Big\} \Big) + \partial_{x} \mathscr{W}_{k-a},
%\end{align*}
which satisfies the way of writing the equations in Theorem \ref{mainResult} and also satisfies
\begin{equation*}
    \frac{\beta_{k}^{a}}{2} = \Gamma_{k,k+a}^{k} = \Gamma_{k+a,k}^{k} = \Gamma_{k,k}^{k+a} = \gamma_{k+a}^{a},
\end{equation*}
\begin{equation*}
    \lambda_{k}^{k-a,k-a'} = \Gamma_{k-a,k-a'}^{k} = \lambda_{k}^{k-a',k-a} = \Gamma_{k-a',k-a}^{k} = \lambda_{k-a}^{k,k-a'} = \Gamma_{k,k-a'}^{k-a}
\end{equation*}
which imply conditions \eqref{betaGamma}-\eqref{lambda}. \qedhere
\end{proof}

%%%%%%%%%%%%%%%%%%%%%%%%%%%%%%%%%%%%%%%%%%%%%%%%%%%%%%%%%%%%%%%
%%%%%%%%%%%%%%%%%%%%%%%%%%%%%%%%%%%%%%%%%%%%%%%%%%%%%%%%%%%%%%%
%%%%%%%%%%%%%%%%%%%%%%%%%%%%%%%%%%%%%%%%%%%%%%%%%%%%%%%%%%%%%%%
%%%%%%%%%%%%%%%%%%%%%%%%%%%%%%%%%%%%%%%%%%%%%%%%%%%%%%%%%%%%%%%

\section{Proof of Lemma \ref{ujBj-djBj=0} }
\label{sec:summation-lemma}

%%%%%%%%%%%%%%%%%%%%%%%%%%%%%%%%%%%%%%%%%%%%%%%%%%%%%%%%%%%%%%%
%%%%%%%%%%%%%%%%%%%%%%%%%%%%%%%%%%%%%%%%%%%%%%%%%%%%%%%%%%%%%%%
%%%%%%%%%%%%%%%%%%%%%%%%%%%%%%%%%%%%%%%%%%%%%%%%%%%%%%%%%%%%%%%
%%%%%%%%%%%%%%%%%%%%%%%%%%%%%%%%%%%%%%%%%%%%%%%%%%%%%%%%%%%%%%%

\begin{proof}[Proof of Lemma \ref{ujBj-djBj=0}]
First, we will prove that
\begin{align*}
    \partial_{k,j}B_{k,j}(u) &= \partial_{k,j} \Big(\alpha_{k} (w_{k,j} - w_{k,j-1}) + \sum_{l \neq 0} \beta_{k}^{l} (b_{k,j}^{l} - b_{k,j-1}^{l}) + \sum_{l \neq 0} \gamma_{k}^{l} (r_{k,j}^{l} - r_{k,j-1}^{l}) \\
    & \quad +\sum_{l \neq 0} \sum_{\substack{l' \neq 0 \\ l' \neq l}} \lambda_{k}^{k-l,k-l'} (p_{k,j}^{l,l'} - p_{k,j-1}^{l,l'}) \Big) = 0
\end{align*}
Recalling the definition of $w$, we can see that
\begin{align*}
        \sum_{j \in \Z_{M}} \partial_{k,j}(w_{k,j} - w_{k,j-1}) 
        = 
        \frac{1}{3} \sum_{j \in \Z_{M}} (u_{k,j+1} - u_{k,j-1})=0.
\end{align*}
Next, recalling the definitions of $b,\, r$ and $p$, we can see that
\begin{eqnarray*}
	\partial_{k,j} (b_{k,j}^{l} - b_{k,j-1}^{l})
	=
	\partial_{k,j} (r_{k,j}^{l} - r_{k,j-1}^{l})
	=
	\partial_{k,j}(p_{k,j}^{l,l'} - p_{k,j-1}^{l,l'}) 
	=
	0. 
\end{eqnarray*}
We are left to prove that
\begin{align*}
     \sum_{k \in \Z_{K}} \sum_{j \in \Z_{M}} u_{k,j} B_{k,j} &= \sum_{k \in \Z_{K}} \sum_{j \in \Z_{M}} u_{k,j} \Big(\alpha_{k} (w_{k,j} - w_{k,j-1}) + \sum_{l \neq 0} \beta_{k}^{l} (b_{k,j}^{l} - b_{k,j-1}^{l})  \\
    & \quad + \sum_{l \neq 0} \gamma_{k}^{l} (r_{k,j}^{l} - r_{k,j-1}^{l}) +\sum_{l \neq 0} \sum_{\substack{l' \neq 0 \\ l' \neq l}} \lambda_{k}^{k-l,k-l'} (p_{k,j}^{l,l'} - p_{k,j-1}^{l,l'}) \Big) = 0.
\end{align*}
We will divide the proof in several steps. In each one of them, we will highlight terms that produce telescopic sums.
\begin{itemize}
    \item $\displaystyle \sum_{j \in \Z_{M}} u_{k,j} (w_{k,j} - w_{k,j-1}) = 0 \ \forall k \in \Z_{K}$: by the definition of $w$,
    \begin{align*}
        u_{k,j}(w_{k,j} - w_{k,j-1}) & = \frac{1}{3}u_{k,j}(u_{k,j}^{2} + u_{k,j}u_{k,j+1} + u_{k,j+1}^{2}) \\
        & \quad - \frac{1}{3}u_{k,j}(u_{k,j-1}^{2} + u_{k,j-1}u_{k,j} + u_{k,j}^{2}) \\
        & = \frac{1}{3}(u_{k,j}^{2}u_{k,j+1} - u_{k,j-1}^{2}u_{k,j}) + \frac{1}{3}(u_{k,j}u_{k,j+1}^{2} - u_{k,j}^{2}u_{k,j-1}).
    \end{align*}
    Both summands yield telescopic sums when summing over $j \in \Z_{M}$.
    
    \vspace{2ex}
  	
  	%%%%%%%%%%%%%%%%%%%%%%%%%%%%%%%%%%%%%%%%%%%%%%%%
  	
    \item $\displaystyle \sum_{k \in \Z_{K}} \sum_{j \in \Z_{M}} \Big( \sum_{l \neq 0} u_{k,j} \beta_{k}^{l} (b_{k,j}^{l} - b_{k,j-1}^{l}) + \sum_{l \neq 0} u_{k,j} \gamma_{k}^{l} (r_{k,j}^{l} - r_{k,j-1}^{l}) \Big) = 0$: first, by definition of $b$,
    \begin{align*}
        & u_{k,j}(\beta_{k}^{l}b_{k,j}^{l} - \beta_{k}^{l}b_{k,j-1}^{l}) \\
        & = u_{k,j}(\frac{\beta_{k}^{l}}{2} u_{k,j}u_{k+l,j} + \frac{\beta_{k}^{l}}{2} u_{k,j+1}u_{k+l,j+1}) - u_{k,j}(\frac{\beta_{k}^{l}}{2} u_{k,j-1}u_{k+l,j-1} + \frac{\beta_{k}^{l}}{2}u_{k,j}u_{k+l,j}) \\
        & = \underbrace{\frac{\beta_{k}^{l}}{2} u_{k,j}u_{k,j+1}u_{k+l,j+1}}_{1} - \underbrace{\frac{\beta_{k}^{l}}{2} u_{k,j} u_{k,j-1}u_{k+l,j-1}}_{2}.
    \end{align*}
    Next,
     \begin{align*}
        u_{k,j}(\gamma_{k}^{l}r_{k,j}^{l} - \gamma_{k}^{l}r_{k,j-1}^{l}) 
       =
       \underbrace{\gamma_{k}^{l} u_{k,j} u_{k-l,j}u_{k-l,j+1}}_{3} - \underbrace{ \gamma_{k}^{l} u_{k,j} u_{k-l,j-1}u_{k-l,j}}_{4}.
    \end{align*}
    Now, looking at the terms for $k+l$ instead of $k$, we obtain
    \begin{align*}
        u_{k+l,j}(\gamma_{k+l}^{l}r_{k+l,j}^{l} - \gamma_{k+l}^{l}r_{k+l,j-1}^{l}) 
        = 
        \underbrace{\gamma_{k+l}^{l} u_{k+l,j} u_{k,j}u_{k,j+1}}_{2} - \underbrace{ \gamma_{k+l}^{l} u_{k+l,j} u_{k,j-1}u_{k,j}}_{1}.
    \end{align*}
    Similarly for $k-l$, we obtain
    \begin{align*}
         u_{k-l,j}(\beta_{k-l}^{l}b_{k-l,j}^{l} - \beta_{k-l}^{l}b_{k-l,j-1}^{l}) 
         = 
         \underbrace{\frac{\beta_{k-l}^{l}}{2} u_{k,j+1} u_{k-l,j+1} u_{k-l,j}}_{4} - \underbrace{\frac{\beta_{k-l}^{l}}{2} u_{k,j-1} u_{k-l,j-1} u_{k-l,j}}_{3}.
    \end{align*}
    Terms marked with the same number underneath are telescopic when summing over $j \in \Z_{M}$ by recalling that $\frac{\beta_{k}^{l}}{2} = \gamma_{k+l}^{l}$.
    
    \vspace{2ex}
    
    %%%%%%%%%%%%%%%%%%%%%%%%%%%%%%%%%%%%%%%%%%%%%%%%

    \item $\displaystyle \sum_{k \in \Z_{K}} \sum_{j \in \Z_{M}} \sum_{l \neq 0} \sum_{\substack{l' \neq 0 \\ l' \neq l}} u_{k,j} \lambda_{k}^{k-l,k-l'} (p_{k,j}^{l,l'} - p_{k,j-1}^{l,l'}) = 0$: by definition of $p$,
    \begin{align*}
        & u_{k,j}(\lambda_{k}^{k-l,k-l'}p_{k,j}^{l,l'} - \lambda_{k}^{k-l,k-l'}p_{k,j-1}^{l,l'}) \nonumber \\
        & = u_{k,j} \Bigg( \frac{\lambda_{k}^{k-l,k-l'}}{6} (2u_{k-l,j}u_{k-l',j} + u_{k-l,j}u_{k-l',j+1} + u_{k-l,j+1}u_{k-l',j} + 2u_{k-l,j+1}u_{k-l',j+1}) \Bigg) \nonumber \\
        & \quad - u_{k,j} \Bigg( \frac{\lambda_{k}^{k-l,k-l'}}{6} (2u_{k-l,j-1}u_{k-l',j-1} +u_{k-l,j-1}u_{k-l',j} + u_{k-l,j}u_{k-l',j-1} + 2u_{k-l,j}u_{k-l',j}) \Bigg) \nonumber \\
        %\label{firstLine}
    \end{align*}
    \begin{align*}
        & = u_{k,j} \Bigg( \frac{\lambda_{k}^{k-l,k-l'}}{6} (\underbrace{u_{k-l,j}u_{k-l',j+1} + u_{k-l,j+1}u_{k-l',j}}_{5} + \underbrace{2u_{k-l,j+1}u_{k-l',j+1}}_{6}) \Bigg) \\
        %\label{secondLine}
        & \quad - u_{k,j} \Bigg( \frac{\lambda_{k}^{k-l,k-l'}}{6} (\underbrace{2u_{k-l,j-1}u_{k-l',j-1}}_{5} + \underbrace{u_{k-l,j-1}u_{k-l',j} + u_{k-l,j}u_{k-l',j-1}}_{6}) \Bigg)
    \end{align*}
    Terms marked with 5 will cancel when summing over $j \in \Z_{M}$, $k \in \Z_{K}$, $l \in \Z_{K}$ and $l' \in \Z_{K}$ by recalling that $\lambda_{k}^{k-l,k-l'}= \lambda_{k}^{k-l',k-l} = \lambda_{k-l}^{k,k-l'}$. The same will happen for terms marked with 6. \\
%    For example, the first term in
%    \begin{equation}
%    \label{terms}
%        u_{k_{1},j_{1}} \lambda_{k_{1}}^{k_{1}-a,k_{1}-a'} u_{k_{1}-a,j_{1}}u_{k_{1}-a',j_{1}+1} +  u_{k_{1},j_{1}} \lambda_{k_{1}}^{k_{1}-a,k_{1}-a'} u_{k_{1}-a,j_{1}+1}u_{k_{1}-a',j_{1}}
%    \end{equation}
%    obtained from \eqref{firstLine} with $k = k_{1}$, $j=j_{1}$, $l = a$, $l'=a'$ will cancel with the terms
%    \begin{equation*}
%        2 u_{k_{1}-a',j_{1}+1} \lambda_{k_{1}-a'}^{k_{1},k_{1}-a} u_{k_{1},j_{1}} u_{k_{1}-a,j_{1}},
%    \end{equation*}
%    \begin{equation*}
%        2 u_{k_{1}-a',j_{1}+1} \lambda_{k_{1}-a'}^{k_{1}-a,k_{1}} u_{k_{1}-a,j_{1}} u_{k_{1},j_{1}},
%    \end{equation*}
%    from \eqref{secondLine} with $k = k_{1}-a'$, $j=j_{1}+1$, $l = -a'$, $l'=a-a'$ and $k = k_{1}-a'$, $j=j_{1}+1$, $l = a-a'$, $l'=-a'$, respectively.
%    The second term in \eqref{terms} will cancel with the terms
%    \begin{equation*}
%        2 u_{k_{1}-a,j_{1}+1} \lambda_{k_{1},k_{1}-a'}^{k_{1}-a} u_{k_{1},j_{1}} u_{k_{1}-a',j_{1}},
%    \end{equation*}
%    \begin{equation*}
%        2 u_{k_{1}-a,j_{1}+1} \lambda_{k_{1}-a',k_{1}}^{k_{1}-a} u_{k_{1}-a',j_{1}} u_{k_{1},j_{1}},
%    \end{equation*}
%    from \eqref{secondLine} with by $k = k_{1}-a$, $j=j_{1}+1$, $l = -a$, $l'=a'-a$ and $k = k_{1}-a'$, $j=j_{1}+1$, $l = a'-a$, $l'=-a$, respectively. \qedhere
\end{itemize}
\end{proof}

\bibliographystyle{siam}
\bibliography{bibliography}

%%%%%%%%%%%%%%%%%%%%%%%%%%%%%%%%%%%%%%%%%%%%%%%%%%%%%%%%%%%%%%%
%%%%%%%%%%%%%%%%%%%%%%%%%%%%%%%%%%%%%%%%%%%%%%%%%%%%%%%%%%%%%%%
%%%%%%%%%%%%%%%%%%%%%%%%%%%%%%%%%%%%%%%%%%%%%%%%%%%%%%%%%%%%%%%
%%%%%%%%%%%%%%%%%%%%%%%%%%%%%%%%%%%%%%%%%%%%%%%%%%%%%%%%%%%%%%%
%%%%%%%%%%%%%%%%%%%%%%%%%%%%%%%%%%%%%%%%%%%%%%%%%%%%%%%%%%%%%%%
%%%%%%%%%%%%%%%%%%%%%%%%%%%%%%%%%%%%%%%%%%%%%%%%%%%%%%%%%%%%%%%
%%%%%%%%%%%%%%%%%%%%%%%%%%%%%%%%%%%%%%%%%%%%%%%%%%%%%%%%%%%%%%%
%%%%%%%%%%%%%%%%%%%%%%%%%%%%%%%%%%%%%%%%%%%%%%%%%%%%%%%%%%%%%%%

\end{document}